\documentclass[11pt]{article}
\usepackage{color}
\usepackage[margin=1in]{geometry}
\usepackage{amssymb}
\usepackage{amsmath}
\usepackage{mathrsfs}
\usepackage{amsthm}
\usepackage{yhmath}
\usepackage{graphicx}
\usepackage{cases}
\usepackage{cite}
\usepackage{authblk}
\usepackage{algorithm,algorithmic}

\newtheorem{theorem}{Theorem}
\newtheorem{lemma}{Lemma}
\newtheorem{definition}{Definition}

\newtheorem{remark}{Remark}
\newtheorem{assumption}{Assumption}
\newcommand{\R}{\mathbb{R}}

\newcommand{\norm}[1]{\left\lVert#1\right\rVert}
\newcommand{\itp}{\hat{\otimes}_\varepsilon}

\title{A deformation-based framework for learning solution mappings of PDEs defined on varying domains}

\author[1]{Shanshan Xiao}
\author[2,3,*]{Pengzhan Jin}
\author[4,5]{Yifa Tang}

\affil[1]{Department of Mathematical Sciences, Tsinghua University, Beijing 100084, China}
\affil[2]{National Engineering Laboratory for Big Data Analysis and Applications, Peking University, Beijing 100871, China}
\affil[3]{Chongqing Research Institute of Big Data, Peking University, Chongqing 401329, China}
\affil[4]{State Key Laboratory of Mathematical Sciences, Academy of Mathematics and Systems Science, Chinese Academy of Sciences, Beijing 100190, China}
\affil[5]{School of Mathematical Sciences, University of Chinese Academy of Sciences, Beijing 100049, China}
\affil[*]{Corresponding author. E-mail: jpz@pku.edu.cn}

\date{}

\begin{document}
	
\maketitle
\begin{abstract}
In this work, we establish a deformation-based framework for learning solution mappings of PDEs defined on varying domains. The union of functions defined on varying domains can be identified as a metric space according to the deformation, then the solution mapping is regarded as a continuous metric-to-metric mapping, and subsequently can be represented by another continuous metric-to-Banach mapping using two different strategies, referred to as the D2D subframework and the D2E subframework, respectively. We point out that such a metric-to-Banach mapping can be learned by neural networks, hence the solution mapping is accordingly learned. With this framework, a rigorous convergence analysis is built for the problem of learning solution mappings of PDEs on varying domains. As the theoretical framework holds based on several pivotal assumptions which need to be verified for a given specific problem, we study the star domains as a typical example, and other situations could be similarly verified. There are four important features of this framework: (1) The domains under consideration are not required to be diffeomorphic, therefore a wide range of regions can be covered by one model provided they are homeomorphic. (2) The deformation mapping is unnecessary to be continuous, thus it can be flexibly established via combining a primary identity mapping and a local deformation mapping. This capability facilitates the resolution of large systems where only local parts of the geometry undergo change. (3) In fact, the recent methods (Geo-FNO, DIMON, etc.) belong to the D2D subframework. We point out that the D2D subframework introduces regularity issues, whereas the proposed D2E subframework remains free from such problems. From a comprehensive perspective, the D2E subframework is better. (4) If a linearity-preserving neural operator such as MIONet is adopted, this framework still preserves the linearity of the surrogate solution mapping on its source term for linear PDEs, thus it can be applied to the hybrid iterative method. We finally present several numerical experiments to validate our theoretical results.
\end{abstract}

\section{Introduction}

Over the past few years, scientific machine learning has garnered significant acclaim for its remarkable achievements in the fields of computational science and engineering \cite{karniadakis2021physics}. Due to the powerful approximation ability of neural networks (NNs) \cite{e2019barron,hanin2019universal,hornik1989multilayer,hornik1990universal,siegel2022high}, a variety of approaches have been put forth to tackle partial differential equations (PDEs). These methods involve parameterizing the solutions through NNs and constructing loss functions based on the strong or variational forms of the PDEs, such as the PINNs \cite{cai2021physics,lu2021physics,pang2019fpinns,raissi2019physics}, the deep Ritz method \cite{yu2018deep}, and the deep Galerkin method \cite{sirignano2018dgm}. A notable drawback of these methods is their slow solving speed, as the NN requires retraining whenever the PDE parameters change. To rapidly obtain solutions for parametric PDEs, several end-to-end approaches known as neural operators have been introduced. These methods aim to directly learn the solution operators of PDEs, bypassing the need for repeated training. The DeepONet \cite{lu2019deeponet,lu2021learning} was firstly proposed in 2019, which employs a branch net and a trunk net to encode the input function and the solution respectively, enabling swift predictions for parametric PDEs. In 2020, the FNO \cite{li2020neural,li2020fourier} was proposed, which learns the solution operators with an integral kernel parameterized in the Fourier domain. Within the same thematic area, numerous studies have emerged \cite{gupta2021multiwavelet,he2023mgno,jin2022mionet,rahman2022u,raonic2023convolutional,wu2023solving}, significantly advancing the field of neural operators. Among the neural operators, the MIONet \cite{jin2022mionet} plays an important role in dealing with complicated cases with multiple inputs, which generalizes the theory and the architecture of the DeepONet. Both the DeepONet and the MIONet are in fact derived from the tensor product of Banach spaces. Furthermore, the trunk nets in the ONets series offer the convenience of differentiating the output functions. This feature allows for the training of neural operators without the need for data, by leveraging the inherent information contained within the PDEs. This approach is being explored in the context of physics-informed DeepONet and MIONet frameworks \cite{wang2021learning,zheng2023state}.

One of the primary limitations of these methods is their specificity in addressing PDEs that are confined to static domains. Given that a multitude of real-world PDE challenges encompass dynamic and diverse domains, there is a pressing requirement to enhance the capacity of neural operators to handle PDE problems across a spectrum of varying regions. In 2022, \cite{goswami2022deep} has leveraged the DeepONet framework to accommodate PDEs characterized by distinct geometric domains. They facilitate the transference of solutions from one domain to another through the mechanism of transfer learning. Nonetheless, this approach still necessitates a retraining phase for the model whenever the domain of the PDE undergoes modification. In the realm of PDEs defined over varying domains, Geo-FNO \cite{li2022fourier} stands out as an innovative extension of the FNO framework. Geo-FNO amplifies the capacity to handle varying domains by employing diffeomorphisms to map the physical space onto a regular computational space. Such a deformation-based strategy is also studied for ONets in the previous version of this work \cite{xiao2024learning} as well as DIMON \cite{yin2024dimon}, while \cite{xiao2024learning} only requires that the domains are homeomorphic and provides related theories. Besides these deformation-based methods, a variety of alternative strategies have been developed to address the challenge of varying domains. For instance, CORAL \cite{serrano2023operator} extracts latent codes for input functions using the auto-decoder \cite{park2019deepsdf} and the Implicit Neural Representations (INR) \cite{sitzmann2020implicit}. Meanwhile, Geom-DeepONet \cite{he2024geom}, 3D-GeoCA \cite{deng2024geometryguided} and PCNO \cite{zeng2025point} leverage the point cloud \cite{qi2017pointnet,qi2017pointnet++} to encode the geometric information. Among these methods, the deformation-based strategy imposes more constraints on the domains. Nonetheless, it is anticipated to perform more effectively when the domains in problem are easily parameterized and undergo only minor changes. While the deformation-based studies concentrate on methodologies and experimental results, there is an urgent need for a comprehensive theoretical analysis and a systematic framework. It is imperative that we establish the correctness and efficacy of this approach through theoretical means.

\textbf{Contributions.} We list the contributions of this work as follows:
\begin{itemize}
    \item We systematically present a deformation-based framework designed to learn solution mappings for PDEs that are defined across a spectrum of varying domains, with a rigorous convergence analysis built, thereby filling a critical theoretical gap in the field.
    \item It does not mandate that the considered domains be diffeomorphic, allowing a single model to encompass a broad array of regions, provided they are homeomorphic, which significantly expands the applicability of the deformation-based learning methods.
    \item The deformation mapping is not required to be continuous, enabling flexible establishment through a combination of a primary identity mapping and localized deformation mappings. This capability facilitates the resolution of large systems where only local parts of the geometry undergo change.
    \item We establish two subframeworks (D2D and D2E), while the recent methods (Geo-FNO, DIMON, etc.) belong to the D2D subframework. We point out that the D2D subframework introduces regularity issues, whereas the proposed D2E subframework remains free from such problems. From a comprehensive perspective, the D2E subframework is better.
    \item If a linearity-preserving neural operator, such as MIONet, is utilized, the framework maintains the linearity of the surrogate solution mapping with respect to its source term for linear PDEs, making it suitable for application in hybrid iterative methods \cite{hu2024hybrid}.
\end{itemize}

This paper is organized as follows. We first briefly introduce the problem setup in Section \ref{sec:setup}. In Section \ref{sec:framework}, we establish the theoretical foundation of our work, and subsequently propose the framework in detail. Section \ref{sec:theoretical_examples} rigorously analyzes the case of the Poisson equation on star domains based on the previous framework. Section \ref{sec:experiment} presents the results of numerical experiments, where we evaluate the methods in several different situations. Finally, we summarize our findings in Section \ref{sec:conclusions}.

\section{Problem setup}\label{sec:setup}

This research originates from the limitations of the early operator learning methods such as DeepONet and FNO which mainly learn the solution operators of PDEs defined on fixed domains. In practical scenarios, PDE problems often involve varying domains.

Basically, the methods we discuss can be applied to both linear and nonlinear PDEs. However, for the sake of clarity and understanding, we consider the following linear problem as
\begin{equation}\label{eq:full_problem_1}
	\begin{cases}
		\mathcal{L}(k^1,k^2,...,k^m)u=f &\quad \mbox{in} \; \Omega,
		\\
		\Gamma u = g &\quad \mbox{on} \; \partial\Omega,
	\end{cases}
\end{equation}
where $\mathcal{L}(k^1,...,k^m)$ is a linear differential operator depending on coefficient functions or constants $k^1,...,k^m$, and $\Gamma$ is a linear boundary operator. The mapping from the inputs to the corresponding solution of such a PDE defined on varying domains can be written as
\begin{equation}
k^1_\Omega\times \cdots \times k^m_\Omega\times f_\Omega\times g_{\partial\Omega}\mapsto u_{\Omega},
\end{equation}
where $k^{i}_\Omega,f_\Omega,u_\Omega$ and $g_{\partial\Omega}$ are functions defined on $\Omega$ and $\partial\Omega$, respectively. 
Taking the Poisson equation as an example:
\begin{equation}\label{eq:poisson_problem}
	\begin{cases}
		-\nabla\cdot(k\nabla u)=f &\quad \mbox{in} \; \Omega,
		\\
		u = g &\quad \mbox{on} \; \partial\Omega,
	\end{cases}	
\end{equation}
the solution mapping is
\begin{equation}\label{eq:full_problem_2}
    k_{\Omega}\times f_{\Omega} \times g_{\partial\Omega} \mapsto u_{\Omega}.
\end{equation}
Note that here $\Omega$ is not fixed, i.e., different tasks provide different $\Omega$. Now we consider the union of functions defined on different domains as $X$, then $X$ will not be a Banach space, so that we cannot directly employ the neural operators to learn this end-to-end map. However, such a space $X$ could be equipped with an appropriate metric that deduces some necessary properties for approximation.

In order to facilitate the readers to understand, here we first consider the simplified case
\begin{equation}
	\begin{cases}
		-\Delta u=f &\quad \mbox{in} \; \Omega,
		\\
		u = 0 &\quad \mbox{on} \; \partial\Omega,
	\end{cases}	
\end{equation}
with the solution mapping
\begin{equation}\label{eq:poisson_mapping}
    f_{\Omega}\mapsto u_{\Omega}.
\end{equation}
The case of (\ref{eq:full_problem_1}-\ref{eq:full_problem_2}) will be of no difficulty as long as this simplified case is solved. The details of the fully-parameterized case are shown in Appendix \ref{sec:mul_in_out} and Section \ref{sec:smooth_star_domain}.

Next we develop the theoretical framework and the method for learning such mappings based on deformation.

\section{A general deformation-based framework}\label{sec:framework}
We establish a general framework with some weaker assumptions compared to the setting of current neural operators which map from Banach spaces to Banach spaces. Firstly, it is necessary to weaken the requirement for the space of input functions. Under the setting of varying domains, two different input functions are usually defined on two different domains, thus the action of sum of these two functions are undefined. Therefore, the input function space is not a linear space, let alone a Banach space. To derive the approximation property, we need the input function space to have at least some kind of metric.

Assume that $U$ is a class of bounded domains in $\R^d$, and let
\begin{equation}
    X:=\bigsqcup_{\Omega\in U}\mathcal{B}(\Omega),\quad \mathcal{B}(\Omega)=\{\rm bounded\ Borel\ functions\ on\ \Omega\}.
\end{equation}
$X$ is exactly the union of the bounded Borel function spaces defined on domains in $U$. We expect to determine the distance between two functions in $X$, especially those with different domains.

\begin{assumption}[metric assumption]\label{ass:metric}
    Suppose that the space $U$ is equipped with a metric $d_{U}$, such that $(U,d_{U})$ is a metric space.
\end{assumption}

\begin{assumption}[deformation assumption]\label{ass:deformation}
    Suppose that there is a standard domain $\Omega_0$ together with corresponding deformation mappings $\mathcal{D}[\Omega]$ for $\Omega$ in $U$, where $\mathcal{D}[\Omega]$ is a bijective between $\Omega_0$ and $\Omega$, $\mathcal{D}[\Omega]\in \mathcal{B}(\Omega_0,\R^d)$, $\mathcal{D}[\Omega]^{-1}\in \mathcal{B}(\Omega,\R^d)$, and the deformation system
\begin{equation}
\begin{split}
    \mathcal{D}:U&\to\mathcal{B}(\Omega_0,\R^d)\subset L^2(\Omega_0,\R^d) \\
    \Omega&\mapsto\mathcal{D}[\Omega]
\end{split}
\end{equation}
is continuous under Assumption \ref{ass:metric}. Note that (i) $\mathcal{D}[\Omega]^{-1}\circ\mathcal{D}[\Omega]$ and $\mathcal{D}[\Omega]\circ\mathcal{D}[\Omega]^{-1}$ are the identity mappings on $\Omega_0$ and $\Omega$ respectively, (ii) we only require the deformation system $\mathcal{D}$ to be continuous, but the deformation mapping $\mathcal{D}[\Omega]$ is allowed to be discontinuous. 
\end{assumption}
\begin{remark}
$\mathcal{D}[\Omega_0]$ usually degenerates to the identity mapping on $\Omega_0$, but this is theoretically not necessary.
\end{remark}

Based on this assumption, the metric for $X$ can be identified.
\begin{definition}[metric for $X$]
\label{def:X_metric}
Under Assumption \ref{ass:metric} and \ref{ass:deformation}, there is a metric $d_X(\cdot,\cdot)$ for $X$ defined as
\begin{equation}
    d_X(f_{\Omega_1},f_{\Omega_2}):=d_{U}(\Omega_1,\Omega_2)+\norm{f_{\Omega_1}\circ\mathcal{D}[\Omega_1]-f_{\Omega_2}\circ\mathcal{D}[\Omega_2]}_{L^2(\Omega_0)},
\end{equation}
where $\Omega_1,\Omega_2\in U$, $f_{\Omega_1}\in\mathcal{B}(\Omega_1)\subset X,f_{\Omega_2}\in\mathcal{B}(\Omega_2)\subset X$.
\end{definition}

To numerically deal with the input functions, we have to discretize such a metric space with the following assumption. Below $\mathbb{N}^*$ denotes the set of natural numbers.

\begin{assumption}[discretization assumption]\label{ass:projection}
Let $X$ be a metric space with metric $d_X(\cdot,\cdot)$, assume that $\kappa:\mathbb{N}^*\to \mathbb{N}^*$ is a strictly monotonically increasing function, $\{\phi_n\}$ and $\{\psi_n\}$ are two sets of mappings, with $\phi_n\in C(X, \R^{\kappa(n)})$, $\psi_n\in C(\phi_n(X), X)$, $P_n:=\psi_n\circ\phi_n\in C(X, X)$, satisfying
\begin{equation}
    \lim_{n\to\infty}\sup_{x\in K}d_X(x,P_n(x))=0,
\end{equation}
for any compact $K\subset X$. We say $\{\phi_n\}$ is a discretization (encoder) for $X$, and $\{\psi_n\}$ is a reconstruction (decoder) for $X$, $P_n$ is the corresponding projection mapping.
\end{assumption}
For a deeper understanding of this assumption, readers may refer to the example of star domains in Section \ref{sec:star_domains} and Figure \ref{fig:encoder}. This assumption is prepared for stating the approximation theorem for learning metric-to-Banach mappings. 
\begin{theorem}[approximation theorem for metric-to-Banach mappings]\label{thm:app_metric_to_banach}
    Let $X$ be a metric space satisfying Assumption \ref{ass:projection} with the projection mapping $P_q=\psi_q\circ\phi_q$ and $\kappa(q)$, $Y$ be a Banach space, $K$ be a compact set in $X$. Suppose that
    \begin{equation}\label{eq:continuous_solution_mapping}
    \mathcal{G}: K\rightarrow Y   
    \end{equation}
    is a continuous mapping, then for any $\epsilon>0$, there exist positive integers $p,q$, a continuous mapping $\tilde{\mathcal{G}}:\R^{\kappa(q)}\to\R^{p}$ and $u\in C(\R^p,Y)$ such that
    \begin{equation}\label{eq:app_general}
    \sup_{v\in K} \norm{\mathcal{G}(v) - u\circ\tilde{\mathcal{G}}(\phi_{q}(v))}_Y< \epsilon.
    \end{equation}
    For convenience, we simplify the notation as
    \begin{equation}
    \begin{split}
        \hat{\mathcal{G}}:K&\to Y \\
        v&\mapsto u\circ\tilde{\mathcal{G}}(\phi_{q}(v)),
    \end{split}
    \end{equation}
    thus Eq. \eqref{eq:app_general} can be rewritten as
    \begin{equation}
        \norm{\mathcal{G}-\hat{\mathcal{G}}}_{C(K,Y)}<\epsilon.
    \end{equation}
\end{theorem}
\begin{proof}
The proof can be found in Appendix \ref{sec:proof_app_metric_to_banach}.
\end{proof}
According to two different strategies for representing solution mappings of PDEs as metric-to-Banach mappings, there are two subframeworks.

\subsection{Representation based on deformed solution space}\label{sec:rep_deformed}

Return to the problem of varying domains. We aim to learn a continuous mapping
\begin{equation}
\begin{split}
    X&\\
    \cup& \\
    \mathcal{H}: K&\longrightarrow  X=\bigsqcup_{\Omega\in U}\mathcal{B}(\Omega),\\
    f_{\Omega}&\longmapsto u_{\Omega}
\end{split}
\end{equation}
where $K\subset X$ is compact, and $\mathcal{H}$ preserves the domain, i.e.,
\begin{equation}\label{eq:preserve_domain}
    \mathcal{H}(\mathcal{B}(\Omega))\subset\mathcal{B}(\Omega),\quad \forall\Omega\in U.
\end{equation}
In accordance with Assumption \ref{ass:deformation} and Definition \ref{def:X_metric}, any function in $X$ can be represented as a domain and a reference function on the standard domain $\Omega_0$, as
\begin{equation}\label{eq:sigma}
\begin{split}
    \sigma:X&\longrightarrow U\times \mathcal{B}(\Omega_0), \\
    f_\Omega&\longmapsto (\Omega, f_\Omega\circ\mathcal{D}[\Omega])
\end{split}
\end{equation}
here $\sigma$ is an isometry. With the following two projections
\begin{equation}\label{eq:projections}
\begin{split}
    \pi_{1}:  U \times  \mathcal{B}(\Omega_{0}) & \rightarrow U,\quad\quad \pi_{2}:  U \times  \mathcal{B}(\Omega_{0}) \rightarrow \mathcal{B}(\Omega_{0}), \\
             (\Omega, f)  & \mapsto \Omega\quad\quad\quad\quad\quad\quad\  (\Omega, f)\mapsto f  \\
\end{split}
\end{equation}
$\mathcal{H}$ can be rewritten as
\begin{equation}
\mathcal{H}=\sigma^{-1}\circ(\pi_1\circ\sigma,\hat{\mathcal{H}}),
\end{equation}
by Eq. \eqref{eq:preserve_domain}, where
\begin{equation}
    \hat{\mathcal{H}}:=\pi_2\circ\sigma\circ\mathcal{H}.
\end{equation}
The problem has now been transformed into learning a metric-to-Banach mapping
\begin{equation}
\begin{split}
    \hat{\mathcal{H}}:K&\longrightarrow\mathcal{B}(\Omega_0)\subset L^2(\Omega_0), \\
    f_\Omega&\longmapsto\pi_2(\sigma(\mathcal{H}(f_\Omega)))=u_\Omega\circ\mathcal{D}[\Omega]
\end{split}
\end{equation}
which is addressed according to Theorem \ref{thm:app_metric_to_banach}. We summarize above analysis as follows.
\begin{theorem}[representation formula based on the deformed solution space]\label{thm:representation_deformed}
    Assume that $(X=\bigsqcup_{\Omega\in U}\mathcal{B}(\Omega),d_X)$ is the metric space defined in Definition \ref{def:X_metric}, $K\subset X$ is a compact set, then for any continuous mapping
    \begin{equation}
        \mathcal{H}:K\longrightarrow X,
    \end{equation}
    satisfying $\mathcal{H}(\mathcal{B}(\Omega))\subset\mathcal{B}(\Omega)$ for $\Omega\in U$, there exists a unique continuous mapping
\begin{equation}
\begin{split}
    \hat{\mathcal{H}}:K&\longrightarrow L^2(\Omega_0), \\
    f_\Omega&\longmapsto\pi_2(\sigma(\mathcal{H}(f_\Omega)))
\end{split}
\end{equation}
    such that
    \begin{equation}
        \mathcal{H}=\sigma^{-1}\circ(\pi_1\circ\sigma,\hat{\mathcal{H}}).
    \end{equation}
\end{theorem}
\begin{proof}
    It is easy to check.
\end{proof}

\subsection{Representation based on extended solution space}\label{sec:rep_extended}
Recall that we expect to learn a continuous mapping
\begin{equation}
\begin{split}
    X&\\
    \cup& \\
    \mathcal{H}: K&\longrightarrow  X=\bigsqcup_{\Omega\in U}\mathcal{B}(\Omega),\\
    f_{\Omega}&\longmapsto u_{\Omega}
\end{split}
\end{equation}
where $K\subset X$ is compact, and $\mathcal{H}(\mathcal{B}(\Omega))\subset\mathcal{B}(\Omega)$ for $\Omega\in U$. As the previous representation strategy in fact learns the mapping $f_\Omega\mapsto u_\Omega\circ\mathcal{D}[\Omega]$, now we consider the strategy of learning $f_\Omega\mapsto u_\Omega$ directly. Denote the projection mapping from $X$ onto $U$ by
\begin{equation}
\begin{split}
    \pi_U:X&\longrightarrow U.\\
    f_\Omega&\longmapsto \Omega
\end{split}
\end{equation}
Suppose that all the domains in $U$ are covered by a $V=[-R,R]^d$ large enough, we can extend the functions $f_\Omega:\Omega\to\R$ to $f_\Omega^V:V\to\R$ with $f_\Omega^V(x)=f_\Omega(x)\ \forall x\in\Omega$, and $f_\Omega^V(x)=0\ \forall x\in V\backslash\Omega$. So that we have defined an extension mapping
\begin{equation}
\begin{split}
    \mathcal{E}:X&\longrightarrow L^2(V).\\
    f_\Omega&\longmapsto f_\Omega^V
\end{split}
\end{equation}
In addition, we define a restriction mapping
\begin{equation}
\begin{split}
    \mathcal{R}:U\times L^2(V)&\longrightarrow X.\\
    (\Omega,f)&\longmapsto f|_{\Omega}
\end{split}
\end{equation}
One can check that $\pi_U$, $\mathcal{E}$ and $\mathcal{R}$ are all continuous (see step 1. in proof of Theorem \ref{thm:continuity_poisson}).
\begin{theorem}[representation formula based on the extended solution space]\label{thm:representation_extended}
    Assume that $(X=\bigsqcup_{\Omega\in U}\mathcal{B}(\Omega),d_X)$ is the metric space defined in Definition \ref{def:X_metric}, $K\subset X$ is a compact set, then for any continuous mapping
    \begin{equation}
        \mathcal{H}:K\longrightarrow X,
    \end{equation}
    satisfying $\mathcal{H}(\mathcal{B}(\Omega))\subset\mathcal{B}(\Omega)$ for $\Omega\in U$, there exists a continuous mapping
    \begin{equation}
        \hat{\mathcal{H}}:K\longrightarrow L^2(V),
    \end{equation}
    such that
    \begin{equation}
        \mathcal{H}=\mathcal{R}\circ(\pi_U,\hat{\mathcal{H}}).
    \end{equation}
\end{theorem}
\begin{proof}
$\hat{\mathcal{H}}:=\mathcal{E}\circ\mathcal{H}$ is a feasible choice. Note that in this theorem the choice of $\hat{\mathcal{H}}$ is not unique.
\end{proof}

\subsection{D2D and D2E subframeworks}

With above results, we show the two subframeworks D2D and D2E for learning the solution mappings of PDE problems defined on varying domains. Here D2D means ``Deformation-based discretization to Deformed solution space'', while D2E means ``Deformation-based discretization to Extended solution space'', corresponding to Section \ref{sec:rep_deformed} and Section \ref{sec:rep_extended}, respectively.

1. Find a suitable standard domain $\Omega_0$ for the domains in $U$, and establish the corresponding deformation system $\mathcal{D}$.

2. Find a suitable discretization (encoder) $\{\phi_n\}$ for $X$. The matched reconstruction (decoder) will be hidden in the neural network and does not need to be explicitly written out.

3. Choose a kind of neural network architecture as the approximator $\tilde{\mathcal{H}}_{\theta}:\R^{\kappa(q)}\to\R^{p}$ together with a 
\begin{equation}
u\in
	\begin{cases}
		C(\R^p,L^2(\Omega_0)), &\quad ({\rm D2D})
		\\
		C(\R^p,L^2(V)), &\quad ({\rm D2E})
	\end{cases}
\end{equation}
as in Theorem \ref{thm:app_metric_to_banach}, and the mapping $u$ may also depend on some trainable parameters such as the trunk net in DeepONet/MIONet. The metric-to-Banach mapping is
\begin{equation}
\hat{\mathcal{H}}_\theta(f_\Omega):=u\circ\tilde{\mathcal{H}}_{\theta}(\phi_{q}(f_\Omega)).
\end{equation}

4. Training stage: Given a dataset $\{f_{\Omega_i},u_{\Omega_i}\}_{i=1}^N$, set the loss function to
\begin{equation}\label{eq:train}
L(\theta)=
	\begin{cases}
		\frac{1}{N}\sum_{i=1}^N\norm{\hat{\mathcal{H}}_\theta(f_{\Omega_i})-u_{\Omega_i}\circ\mathcal{D}[\Omega_i]}_{L^2(\Omega_0)}^2. &\quad ({\rm D2D})
		\\
		\frac{1}{N}\sum_{i=1}^N\norm{\hat{\mathcal{H}}_\theta(f_{\Omega_i})-u_{\Omega_i}}_{L^2(\Omega_i)}^2. &\quad ({\rm D2E})
	\end{cases}
\end{equation}
Optimize $L(\theta)$ and obtain $\theta^*=\arg\min_{\theta}L(\theta)$.

5. Predicting stage: Given a new $f_\Omega$, predict the solution by
\begin{equation}\label{eq:predict}
u_{\Omega}^{\rm pred}=
	\begin{cases}
		\hat{\mathcal{H}}_{\theta^*}(f_{\Omega})\circ\mathcal{D}[\Omega]^{-1}. &\quad ({\rm D2D})
		\\
		\hat{\mathcal{H}}_{\theta^*}(f_{\Omega})|_\Omega. &\quad ({\rm D2E})
	\end{cases}
\end{equation}

Briefly, the primary distinction between D2D and D2E lies in the target spaces: D2D maps to the deformed solution space $L^2(\Omega_0)$, whereas D2E maps to the extended solution space $L^2(V)$.

\textbf{Pros and cons.} One defect of D2D is that the regularity of the predicted solution $u_{\Omega}^{\rm pred}=\hat{\mathcal{H}}_{\theta^*}(f_{\Omega})\circ\mathcal{D}[\Omega]^{-1}$ will be influenced by the deformation mapping $\mathcal{D}[\Omega]$. If $\mathcal{D}[\Omega]^{-1}$ is non-differentiable or discontinuous on a subset $\Gamma\subset\Omega$, then the regularity of $u_{\Omega}^{\rm pred}$ on $\Gamma$ may be worse than $u_\Omega$. Fortunately, D2E subframework overcomes this shortcoming, since the predicted solution is directly expressed by the neural network without involving $\mathcal{D}[\Omega]$. This point will be shown in the experiment part. However, D2E incurs marginally higher computational costs than D2D regarding to both memory and speed, because we have to store the integration points for each $\Omega_i$ during training, which also results in more calculations. From a comprehensive perspective, the D2E subframework is still better.

\textbf{Additional contents.} For clarity, the above discussion is presented for single-input\&output case. The multiple-input\&output case, while conceptually identical, is detailed in Appendix \ref{sec:mul_in_out}. Furthermore, we provide details regarding the possible choices of neural operators adopted in this framework, i.e. MIONet and FNO, in Appendix \ref{sec:choices_no}.

\subsection{A summary of framework}
We summarize the theoretical framework presented.
\begin{enumerate}
    \item Verify that $X:=\bigsqcup_{\Omega\in U}\mathcal{B}(\Omega)$ is a discretable metric space according to deformation.
    \item Verify that the metric-to-metric solution mapping $\mathcal{H}:X\supset K\to X$ of a specific PDE problem is indeed continuous.
    \item Represent $\mathcal{H}$ by a continuous metric-to-Banach mapping $\hat{\mathcal{H}}:K\to L^2(\Omega_0)$ (D2D) or $\hat{\mathcal{H}}:K\to L^2(V)$ (D2E).
    \item Approximate $\hat{\mathcal{H}}$ via the neural network $\hat{\mathcal{H}}_\theta$, thus $\mathcal{H}$ has been learned.
\end{enumerate}
Whether the first two steps are valid depends on the specific shapes of the domains in $U$ as well as the PDE problem. We will show some typical examples under this framework.

\section{Theoretical examples}\label{sec:theoretical_examples}

\subsection{Example of star domains}\label{sec:star_domains}
The above theoretical framework is established based on several pivotal assumptions, including the definition of metric on $U$, the related deformation system, as well as the way of discretization. Given a specific problem, we need to provide relevant definitions and verify that they meet the aforementioned assumptions. The first example focuses on the star domains, and it clearly illustrates how this framework works.

In this example, we discuss the simplified star domains in $\R^d$, which are required to take the centroid as the reference point.
\begin{definition}\label{def:star_domain}
    We define the \textbf{star domains} as follows: Consider the centroid of an open set $\Omega\subset\R^d$ as the original point, then we refer to $\Omega$ as a star domain if its boundary $\partial\Omega$ can be expressed as a Lipschitz continuous function
    \begin{equation}
        b:S\to\R^{+},\quad S:=\{e\in\R^d|\norm{e}_2=1\}.
    \end{equation}
    In such a case $\Omega=\{0\neq x\in\R^d|\norm{x}_2<b(x/\norm{x}_2)\}\cup\{0\}$.
\end{definition}
\noindent Now we consider two spaces. Denote
\begin{equation}
    U:=\{ \Omega\subset \R^d ~|~ \Omega ~ \textrm{is a star domain with a Lipschitz coefficient defined above no more than $L$}\}
\end{equation}
and
\begin{equation}
    X:=\bigsqcup_{\Omega\in U}\mathcal{B}\left(\Omega\right).
\end{equation}
We will provide metrics on these two spaces, a deformation system for $U$, as well as a discretization for $X$.

\textbf{Metric for $U$.} Denote by $c_\Omega$ and $b_\Omega$, the centroid of $\Omega$ and the boundary function defined in Definition \ref{def:star_domain} of $\Omega$, respectively. Let $\Omega_1,\Omega_2\in U$, then we define 
\begin{equation}
    d_{U}(\Omega_1,\Omega_2) := d_{E}(c_{\Omega_1}, c_{\Omega_2}) + \sup_{e\in S}| b_{\Omega_{1}}(e) -  b_{\Omega_{2}}(e)|,
\end{equation}
where $d_E$ is the Euclidean metric. One can readily check that $(U,d_U)$ is a metric space.

\textbf{Deformation system for $U$.} We choose the unit open ball $B(0,1)$ as the standard domain $\Omega_0$. For any $\Omega\in U$, the deformation mapping $\mathcal{D}[\Omega]$ is defined as
\begin{equation}\label{eq:def_star}
    \mathcal{D}[\Omega](x):=c_\Omega+b_\Omega(x/\norm{x}_2)\cdot x,\quad x\in \Omega_0=B(0,1).
\end{equation}
It is obvious that $\mathcal{D}[\Omega]$ is a bounded Borel bijective between $\Omega_0$ and $\Omega$, and the deformation system $\mathcal{D}$ is continuous w.r.t. $\Omega$ due to
\begin{equation}
\begin{split}
    \norm{\mathcal{D}[\Omega_1]-\mathcal{D}[\Omega_2]}_{L^2(\Omega_0,\R^d)}=&\left(\int_{\Omega_0}\norm{(c_{\Omega_1}-c_{\Omega_2})+(b_{\Omega_1}(x/\norm{x}_2)-b_{\Omega_2}(x/\norm{x}_2))\cdot x}_2^2dx\right)^{\frac{1}{2}} \\
    \leq&d_U(\Omega_1,\Omega_2)\cdot\left(\int_{\Omega_0}(1+\norm{x}_2)^2dx\right)^{\frac{1}{2}}.
\end{split}
\end{equation}

\textbf{Discretization for $X$.} As aforementioned, the metric for $X$ is
\begin{equation}
    d_X(f_{\Omega_1},f_{\Omega_2})=d_{U}(\Omega_1,\Omega_2)+\norm{f_{\Omega_1}\circ\mathcal{D}[\Omega_1]-f_{\Omega_2}\circ\mathcal{D}[\Omega_2]}_{L^2(\Omega_0)}.
\end{equation}
Reasonably, we discretize $U$ and $\mathcal{B}(\Omega_0)\subset L^2(\Omega_0)$ separately and finally merge them.
\begin{itemize}
    \item For $U$, there are many strategies of discretization, here we consider a natural approach as
    \begin{equation}\label{eq:discretization_U}
        \phi_n^1(\Omega):=\left(c_\Omega,b_\Omega(e_1^n),b_\Omega(e_2^n),...,b_\Omega(e_n^n)\right)\in\R^{n+d},
    \end{equation}
    where $V_n:=\{e_1^n,...,e_n^n\}\subset S$ are expected to be evenly selected. Let ${\rm conv}(V_n)$ be the convex hull of $V_n$, we decompose the non-simplex $(d-1)$-dimensional polytopes on the boundary of ${\rm conv}(V_n)$ into some small simplices, and refer the ${\rm conv}(V_n)$ with piece-wise simplex boundary as $V_n^s$. Denote by $M_n$ the maximum diameter of the simplices on the boundary of $V_n^s$, then $\{e_j^i\}$ are selected such that
    \begin{equation}\label{eq:lim_Mn}
        \lim_{n\to\infty}M_n=0.
    \end{equation}
    
    The related reconstruction $\psi_n^1$ can be established by mapping $\phi_n^1(\Omega)$ to the polytope which takes
    \begin{equation}
        \{c_\Omega+b_\Omega(e_i^n)\cdot e_i^n\}_{i=1}^n
    \end{equation}
    as its vertexes and has the same simplex partition of the boundary as $V_n^s$. For the 2-d case, we can simply choose
    \begin{equation}
        e_i^n=(\cos(2\pi i/n),\sin(2\pi i/n)),\quad 1\leq i\leq n,
    \end{equation}
    and for the 3-d case, we may adopt the Fibonacci lattice \cite{swinbank2006fibonacci,gonzalez2010measurement} as
    \begin{equation}
        e_i^n=\left(\sqrt{1-z_{n,i}^2}\cos((\sqrt{5}-1)\pi i),\sqrt{1-z_{n,i}^2}\sin((\sqrt{5}-1)\pi i),z_{n,i}\right),\quad 1\leq i\leq n,
    \end{equation}
    where $z_{n,i}=1-(2i-2)/(n-1)$. In practice, the $c_\Omega$ in Eq. \eqref{eq:discretization_U} are usually omitted since the PDEs are translation independent and we can move all the centroids to the original point. In such a simplified case,
    \begin{equation}
        \phi_n^1(\Omega)=\left(b_\Omega(e_1^n),b_\Omega(e_2^n),...,b_\Omega(e_n^n)\right)\in\R^{n}.
    \end{equation}
    To provide intuition for the discretization and the reconstruction, we show an illustration in Figure \ref{fig:encoder}. For clarity, the illustration employs direct coordinate encoding, fundamentally equivalent to the above radius-based encoding.

    \begin{figure}[htbp]
    \centering
    \includegraphics[width=0.99\textwidth]{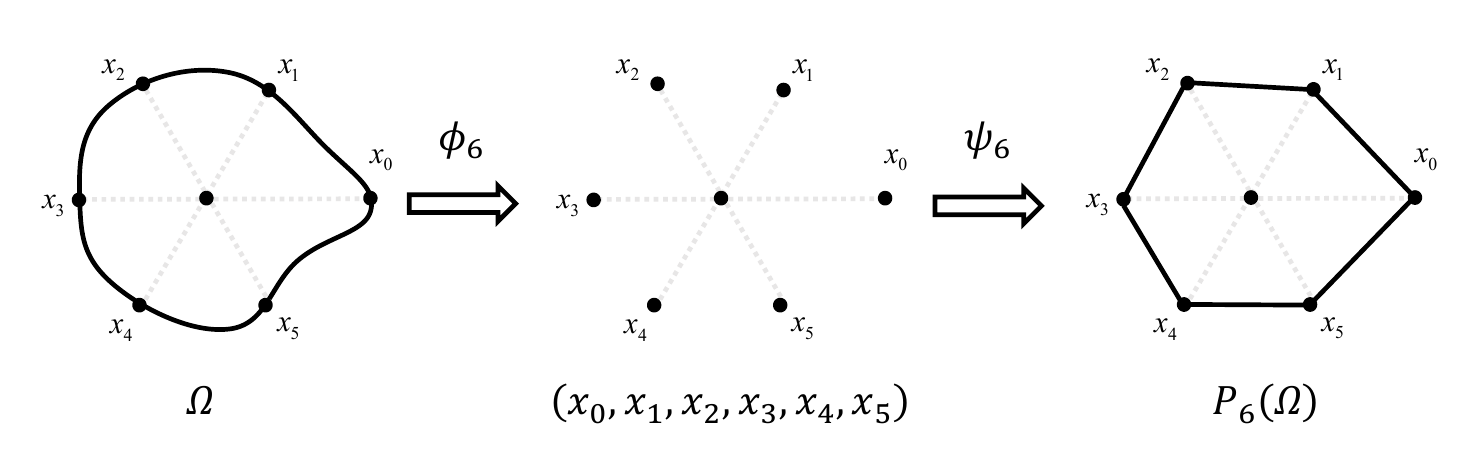}
    \caption{An illustration of the discretization mapping (encoder) and the reconstruction mapping (decoder) for the star domains.}
    \label{fig:encoder}
    \end{figure}
    
    \item For $\mathcal{B}(\Omega_0)$, we first divide $\Omega_0$ into $n$ mutually disjointed subdomains as
    \begin{equation}
        \Omega_0=\bigsqcup_{i=1}^n\Omega_i^n,
    \end{equation}
    and denote by $M_n$ the maximum diameter of these subdomains. Additionally, we ensure that
    \begin{equation}
        \lim_{n\to\infty}M_n=0.
    \end{equation}
    Then the discretization is defined as
    \begin{equation}
        \phi_n^2(f):=\left(\frac{\int_{\Omega_1^n}f(x)dx}{m(\Omega_1^n)},\frac{\int_{\Omega_2^n}f(x)dx}{m(\Omega_2^n)},...,\frac{\int_{\Omega_n^n}f(x)dx}{m(\Omega_n^n)}\right)\in\R^n,\quad f\in\mathcal{B}(\Omega_0),
    \end{equation}
    where $m(\cdot)$ means the Lebesgue measure.
    The related reconstruction $\psi_n^2$ can be established by mapping $\phi_n^2(f)$ to the piece-wise constant function with a value of $\int_{\Omega_i^n}f(x)dx/m(\Omega_i^n)$ over $\Omega_i^n$, $1\leq i\leq n$. If $f$ is known as a piece-wise continuous function, then we can simply apply
    \begin{equation}
        \phi_n^2(f)=\left(f(x_1^n),f(x_2^n),...,f(x_n^n)\right)\in\R^n,
    \end{equation}
    with fixed points $x_i^n\in\Omega_i^n$. In practice, we may choose Fibonacci lattice (2-d and 3-d), or randomly sample $n$ uniformly distributed points in $\Omega_0$ for convenience. 
    \item Finally for $X$, we define the discretization as
    \begin{equation}
        \phi_n(f_\Omega):=\left(\phi_{\kappa_1(n)}^1(\Omega),\phi_{\kappa_2(n)}^2(f_\Omega\circ\mathcal{D}[\Omega])\right)\in\R^{\kappa_1(n)+\kappa_2(n)},\quad f_\Omega\in\mathcal{B}(\Omega),
    \end{equation}
    where $\kappa_1,\kappa_2:\mathbb{N}^*\to \mathbb{N}^*$ can be any strictly monotonically increasing function. The related reconstruction $\psi_n$ is established as
    \begin{equation}
        \psi_n(\phi_n(f_\Omega)):=\sigma^{-1}\left(\psi_{\kappa_1(n)}^1\left(\phi_{\kappa_1(n)}^1(\Omega)\right),\psi_{\kappa_2(n)}^2\left(\phi_{\kappa_2(n)}^2(f_\Omega\circ\mathcal{D}[\Omega])\right)\right),
    \end{equation}
    here $\sigma$ is defined in \eqref{eq:sigma}. In this case $\kappa(n):=\kappa_1(n)+\kappa_2(n)$.
\end{itemize}
Although the discretization for $X$ has been given, we still have to show that such a discretization satisfies Assumption \ref{ass:projection}. Similarly, we firstly verify that the discretizations for $U$ and $\mathcal{B}(\Omega_0)$ satisfy Assumption \ref{ass:projection}, then derive the result for $X$.
\begin{lemma}\label{lem:discretization_for_U}
   The metric space $(U, d_{U})$ and the two sets of mappings $\{\phi_{n}^1\}$ and $\{\psi_{n}^1\}$ defined above satisfy Assumption \ref{ass:projection}, i.e.,
\begin{equation}
 \lim_{n\to\infty}\sup_{\Omega\in K}d_{U}(\Omega, P_{n}^1(\Omega)) = 0
\end{equation}
holds for any compact $K\subset U$, where $P_{n}^1 = \psi_{n}^1\circ\phi_{n}^1$ is the projection mapping.
\end{lemma}
\begin{proof}
The proof can be found in Appendix \ref{sec:proof_discretization_for_U}.
\end{proof}

Next we discuss the discretization for $\mathcal{B}(\Omega_0)$.
\begin{lemma}\label{lem:discretization_for_B}
   The metric space $(\mathcal{B}(\Omega_0), \norm{\cdot-\cdot}_{L^2(\Omega_0)})$ and the two sets of mappings $\{\phi_{n}^2\}$ and $\{\psi_{n}^2\}$ defined above satisfy Assumption \ref{ass:projection}, i.e.,
\begin{equation}
 \lim_{n\to\infty}\sup_{f\in K}\norm{f-P_n^2(f)}_{L^2(\Omega_0)} = 0
\end{equation}
holds for any compact $K\subset\mathcal{B}(\Omega_0)$, where $P_{n}^2 = \psi_{n}^2\circ\phi_{n}^2$ is the projection mapping.
\end{lemma}
\begin{proof}
The proof can be found in Appendix \ref{sec:proof_discretization_for_B}.
\end{proof}

Lastly, we show the result for $X$.
\begin{theorem}\label{thm:discretization_for_X}
   The metric space $(X, d_{X})$ and the two sets of mappings $\{\phi_{n}\}$ and $\{\psi_{n}\}$ defined above satisfy Assumption \ref{ass:projection}, i.e.,
\begin{equation}
 \lim_{n\to\infty}\sup_{f_\Omega\in K}d_{X}(f_\Omega, P_{n}(f_\Omega)) = 0
\end{equation}
holds for any compact $K\subset X$, where $P_{n}=\psi_{n}\circ\phi_{n}$ is the projection mapping.
\end{theorem}
\begin{proof}
    Since the projections $\pi_1,\pi_2$ (defined in Eq. \eqref{eq:projections}) are continuous maps which preserve the compactness, we have compact sets $\pi_1(\sigma(K))\subset U$ and $\pi_2(\sigma(K))\subset\mathcal{B}(\Omega_0)$. With the definition of $d_X$ and $P_n$,
    \begin{equation}
        d_{X}(f_\Omega, P_{n}(f_\Omega))=d_U(\Omega,P_{\kappa_1(n)}^1(\Omega))+\norm{f_\Omega\circ\mathcal{D}[\Omega]-P_{\kappa_2(n)}^2(f_\Omega\circ\mathcal{D}[\Omega])}_{L^2(\Omega_0)},
    \end{equation}
    then we have
    \begin{equation}
        \sup_{f_\Omega\in K}d_{X}(f_\Omega, P_{n}(f_\Omega))\leq\sup_{\Omega\in\pi_1(\sigma(K))}d_U(\Omega,P_{\kappa_1(n)}^1(\Omega))+\sup_{f\in\pi_2(\sigma(K))}\norm{f-P_{\kappa_2(n)}^2(f)}_{L^2(\Omega_0)}.
    \end{equation}
    Lemma \ref{lem:discretization_for_U} and Lemma \ref{lem:discretization_for_B} immediately lead to
    \begin{equation}
        \lim_{n\to\infty}\sup_{f_\Omega\in K}d_{X}(f_\Omega, P_{n}(f_\Omega)) = 0,
    \end{equation}
    thus the proof is finished.
\end{proof}
Up to now, we have completed this deformation-based framework for the star domains, for which the approximation property of discretization is guaranteed. For a specific PDE problem, the solution mapping can be learned through the above framework as long as it is indeed a continuous mapping as defined in Theorem \ref{thm:representation_deformed} and \ref{thm:representation_extended}. At the end of this subsection we present a simple case demonstrating that the solution mapping of the 2-d Poisson equation is continuous.

\textbf{Continuity of Poisson equation on varying star domains.} Consider the Poisson equation
\begin{equation}
	\begin{cases}
		-\Delta u=f &\quad \mbox{in} \; \Omega,
		\\
		u = 0 &\quad \mbox{on} \; \partial\Omega,
	\end{cases}	
\end{equation}
where $\Omega\in\R^2$ is a star domain with smooth boundary $\partial\Omega$. Let
\begin{equation}
    U=\{\Omega\in\R^2|{\rm \Omega\ is\ open\ star\ domain,\ \partial\Omega\ is\ smooth}\},
\end{equation}
and $K$ be a compact set in $\bigsqcup_{\Omega\in U}C\left(\Omega\right)$, thus $K$ is also a compact set in $X=\bigsqcup_{\Omega\in U}\mathcal{B}\left(\Omega\right)$. We aim to show the continuity of the solution mapping
\begin{equation}
\begin{split}
    \mathcal{H}:K&\longrightarrow X.\\
    f_\Omega&\longmapsto u_\Omega
\end{split}
\end{equation}
\begin{theorem}\label{thm:continuity_poisson}
    The solution mapping $\mathcal{H}$ of 2-d Poisson equation defined on varying star domains is continuous on compact set $K\subset \bigsqcup_{\Omega\in U}C\left(\Omega\right)$.
\end{theorem}
\begin{proof}
The proof can be found in Appendix \ref{sec:proof_continuity_poisson}.
\end{proof}
With above theoretical results, the convergency of the proposed methods for 2-d Poisson equations on varying star domains is rigorously guaranteed. 

Lastly, we briefly present the algorithm for learning the solution mapping of Poisson equations defined on varying star domains based on MIONet in Algorithm \ref{alg:star}, together with an illustration in Figure \ref{fig:illustration}. Assume that all the domain centroids have been translated to the origin.

\begin{algorithm}\label{alg:star}
\caption{Learning the solution mapping of Poisson equations defined on varying star domains based on MIONet}
\begin{algorithmic}
\REQUIRE{Dataset $\{f_{\Omega_i},u_{\Omega_i}\}_{i=1}^N$}
\ENSURE{Model $\hat{\mathcal{H}}_{\theta^*}$ for predicting the solution $u_\Omega$ given $f_\Omega$}
\STATE{\textbf{1.} For each $\Omega_i$, compute $$\phi_{q_1}^1(\Omega_i):=(r_1^i,...,r_{q_1}^i)\in\R^{q_1},$$and $r^i_j$ is the length of the line segment where the ray at angle $\frac{2\pi j}{q_1}$ intersects the star domain $\Omega_i$.}
\STATE{\textbf{2.} Fix $q_2$ uniformly distributed points in the unit ball, denoted by $x_1,...,x_{q_2}\in\R^d$.}
\STATE{\textbf{3.} For each $f_{\Omega_i}\circ\mathcal{D}[\Omega_i]$, compute $$\phi_{q_2}^2(f_{\Omega_i}\circ\mathcal{D}[\Omega_i]):=(f_{\Omega_i}\circ\mathcal{D}[\Omega_i](x_1),...,f_{\Omega_i}\circ\mathcal{D}[\Omega_i](x_{q_2}))\in\R^{q_2},$$and the deformation system $\mathcal{D}$ is given by \eqref{eq:def_star}.}
\STATE{\textbf{4.} Initialize the first branch net $\tilde{\mathcal{H}}_{\theta}^1:\R^{q_1}\to\R^{p}$, the second branch net $\tilde{\mathcal{H}}_{\theta}^2:\R^{q_2}\to\R^{p}$, and the trunk net $\tilde{u}_\theta:\R^d\to\R^p$. Set
\begin{equation*}
    \hat{\mathcal{H}}_{\theta}(f_\Omega)(y):=\mathcal{S}(\tilde{\mathcal{H}}_{\theta}^1(\phi_{q_1}^1(\Omega))\odot\tilde{\mathcal{H}}_{\theta}^2(\phi_{q_2}^2(f_\Omega\circ\mathcal{D}[\Omega]))\odot\tilde{u}_\theta(y)),\quad y\in\R^d,
\end{equation*}
where $\odot$ is the Hadamard product (element-wise product) and $\mathcal{S}$ is the summation of all the components of a vector.
}
\STATE{\textbf{5.} Train the loss function defined in \eqref{eq:train}, and obtain $\theta^*$.
}
\STATE{\textbf{6.} Predict the solution $u_\Omega$ given $f_\Omega$ via \eqref{eq:predict} based on the trained model $\hat{\mathcal{H}}_{\theta^*}$.}
\end{algorithmic}
\end{algorithm}

\begin{figure}[htbp]
    \centering
    \includegraphics[width=1.00\textwidth]{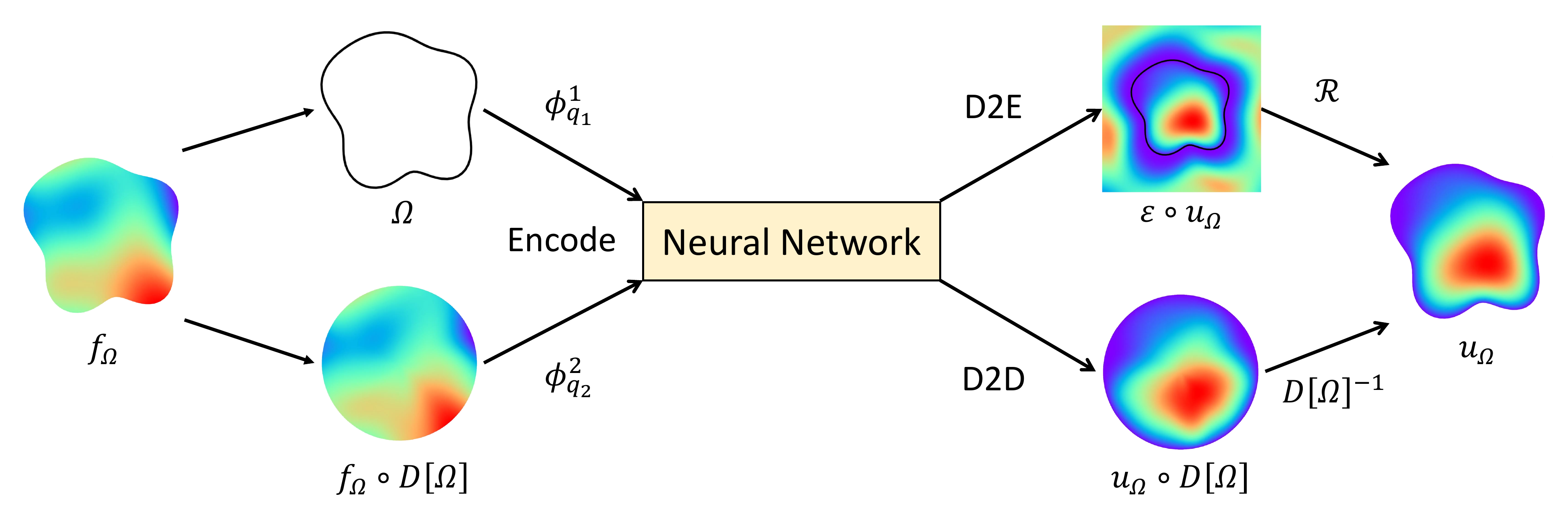}
    \caption{An illustration of the deformation-based framework.}
    \label{fig:illustration}
    \end{figure}

\subsection{Other examples}

As the example of star domains is used to demonstrate the the theoretical procedure, here we further show several examples highlighting some interesting features of this framework. 

\subsubsection{Locally deformed domains: discontinuous deformation mappings}\label{sec:local_deformed}

Consider an important situation that the PDE problems are defined on complex and large domains. In such a case, we may deform only a local part of the whole domain to achieve fast prediction for specific industrial purposes like the inverse design, for which it is hard to find global continuous deformation mappings, much less diffeomorphic ones. Fortunately, it is not difficult to establish piece-wise continuous deformation mappings instead. Here we present a simple example which will be studied in the experiment part. 

Let $\Omega_1:=[0,1]^2$, $\Omega_2(a):=[a-0.15, a+0.15]\times[1, 1.3]$, and
\begin{equation}
    U:=\{\Omega(a)=\Omega_1\cup\Omega_2(a)|a\in [0.3,0.7]\}.
\end{equation}
The metric for $U$ could be
\begin{equation}
    d_U(\Omega(a_1),\Omega(a_2)):=|a_1-a_2|.
\end{equation}
The deformation mapping $\mathcal{D}[\Omega]$ can be established via a moving on $\Omega_2$ as
\begin{equation}
    \mathcal{D}[\Omega(a)](x):=
    \begin{cases}
    x &\quad\quad x\in\Omega_1\\
    x+(a-0.5,0) &\quad\quad x\in\Omega_2(0.5)
    \end{cases},
\end{equation}
which is a bounded Borel bijective between $\Omega_0:=\Omega_1\cup\Omega_2(0.5)$ and $\Omega(a)$, and discontinuous along $\Omega_1\cap\Omega_2(0.5)$. This deformation mapping remains unchanged on the main domain $\Omega_1$, while focusing on the changes in the local domain $\Omega_2$. The related discretization for $X=\bigsqcup_{\Omega\in U}\mathcal{B}(\Omega)$ can be written as
\begin{equation}
    \phi_n(f_{\Omega(a)}):=(a,f_{\Omega(a)}\circ\mathcal{D}[\Omega(a)](x_1),...,f_{\Omega(a)}\circ\mathcal{D}[\Omega(a)](x_n))\in\R^{n+1},
\end{equation}
where $x_1,...,x_n$ are uniformly distributed in $\Omega_0$.

\subsubsection{Composite domains: composite discretization}\label{sec:composite_domain}
A complex domain may be regarded as the composite of several simple domains in the sense of discretization. A typical example is the annulus-like domains, which can be discretized via two star domains. Assume that both outer and inner boundaries of the annulus-like domain can be regarded as the boundaries of the star domains centered at the centroid of the inner region. Define
\begin{equation}
\begin{split}
U:=&\{\Omega\in\R^2|\Omega{\rm \ is\ a\ annulus}-{\rm like\ domain\ with}\\ &{\rm Lipschitz\ coefficients\ of\ outer\ and\ inner\ star\ domains\ no\ more\ than\ }L\}.
\end{split}
\end{equation}
Let $\Omega^{\rm out}$ and $\Omega^{\rm in}$ be the outer and inner star domains of $\Omega$, $c_{\Omega^{\rm in}}$ be the centroid of $\Omega^{\rm in}$, $b_{\Omega^{\rm out}}$ and $b_{\Omega^{\rm in}}$ be the boundary functions centered at $c_{\Omega^{\rm in}}$, then the metric for $U$ can be
\begin{equation}
    d_U(\Omega_1,\Omega_2):=d_{E}(c_{\Omega_1^{\rm in}},c_{\Omega_2^{\rm in}})+\sup_{e\in S}| b_{\Omega_{1}^{\rm in}}(e)-b_{\Omega_{2}^{\rm in}}(e)|+\sup_{e\in S}| b_{\Omega_{1}^{\rm out}}(e)-b_{\Omega_{2}^{\rm out}}(e)|,
\end{equation}
and the deformation $\mathcal{D}[\Omega]$ is defined as
\begin{equation}
\begin{split}
    \mathcal{D}[\Omega](x):=&c_{\Omega^{\rm in}}+(b_{\Omega^{\rm out}}(x/\norm{x}_2) (2\norm{x}_2-1)+b_{\Omega^{\rm in}}(x/\norm{x}_2)(2-2\norm{x}_2))\cdot x/\norm{x}_2,\\ &x\in \Omega_0:=B(0,1)\backslash B(0,0.5).
\end{split}
\end{equation}
The discretization for $X$ is given by
\begin{equation}
\begin{split}
    \phi_n(f_{\Omega}):=&(c_{\Omega^{\rm in}},b_{\Omega^{\rm in}}(e_1^1),...,b_{\Omega^{\rm in}}(e_{\kappa_1(n)}^1),b_{\Omega^{\rm out}}(e_1^2),...,b_{\Omega^{\rm out}}(e_{\kappa_2(n)}^2),\\&f_\Omega\circ\mathcal{D}[\Omega](x_1),...,f_\Omega\circ\mathcal{D}[\Omega](x_{\kappa_3(n)})),
\end{split}
\end{equation}
where $\{e_i^1\},\{e_i^2\}\subset S$, $\{x_i\}\subset\Omega_0$ are uniformly distributed.

\section{Numerical experiments}\label{sec:experiment}
In this section we verify our theoretical results via some numerical examples. Basically, we study the Poisson equation \eqref{eq:poisson_problem}, with several different settings. Note that all the experiments apply the Adam optimizer. The codes are published in GitHub https://github.com/jpzxshi/deformation.

\subsection{Polygonal domains}
We first consider the case of Poisson equation \eqref{eq:poisson_problem} with $k\equiv1$, $g\equiv0$, thus the target mapping is $f_\Omega\mapsto u_\Omega$. We initially generate 1500 convex quadrilaterals (pentagons/hexagons) with related triangular meshes contained within $[0,1]^{2}$, totally 4500 samples. Following this, we generate 4500 random functions on these polygons via Gaussian process, forming the set of random functions as $f_{\Omega}$. We then solve these 4500 equations via finite element method and obtain the set of solutions $u_\Omega$ based on the meshes. There are totally 4500 training data points, together with 1500 test data points generated through the same procedure. We learn this dataset by D2E-MIONet, D2D-MIONet, as well as a benchmark Geo-FNO (O mesh). As for the network size of MIONet, we use a branch net of size [200, 500, 500, 500, 500, 1000] (4 hidden layers with width 500) for the domain, a linear branch net of size [5000, 1000] for the function, as well as a trunk net of size [2, 500, 500, 500, 500, 1000]. We train the MIONet with full-batch. For the benchmark Geo-FNO, we first set the number of layers, the width, and the mode to 6, 64, and 12, respectively, which has a similar number of parameters to the MIONet. Then we train such FNO with different batch sizes 1, 2, 5, 10, 20, 50, 100, 200, 500, 1000, and find the best one (5 is best according to our results). Next we fix the best batch size and test different network sizes of FNO, with the numbers of layers 5, 6, 7, and the widths 32, 64, 128. Each case is exhaustively trained until the test error converges. The final relative $L^2$ errors of D2E-MIONet, D2D-MIONet, Geo-FNO are 3.51\%, 3.46\% and 6.52\%, respectively, shown in Table \ref{tab:errors}. The predictions are presented in Figure \ref{fig:polygonal}. D2E-MIONet and D2D-MIONet exhibit comparable test errors, while D2E-MIONet predicts solutions with better regularity than D2D-MIONet, especially along the line segments connecting vertices to the centroid, where the deformation mappings are not differentiable. Geo-FNO performs worst on both accuracy and regularity, due to its rectangular latent domain which is not diffeomorphic to these polygonal domains.

\begin{table}[htbp]
\centering
\begin{tabular}{|c|c|c|c|}
\hline
Method                & D2E-MIONet & D2D-MIONet & Geo-FNO \\ \hline
Relative $L^2$ error & 3.51\%     & 3.46\%     & 6.52\%  \\ \hline
\end{tabular}
\caption{The relative $L^2$ errors of D2E-MIONet, D2D-MIONet and Geo-FNO for Poisson equations on polygonal domains.}
\label{tab:errors}
\end{table}

\begin{figure}[htbp]
\centering
\includegraphics[width=0.99\textwidth]{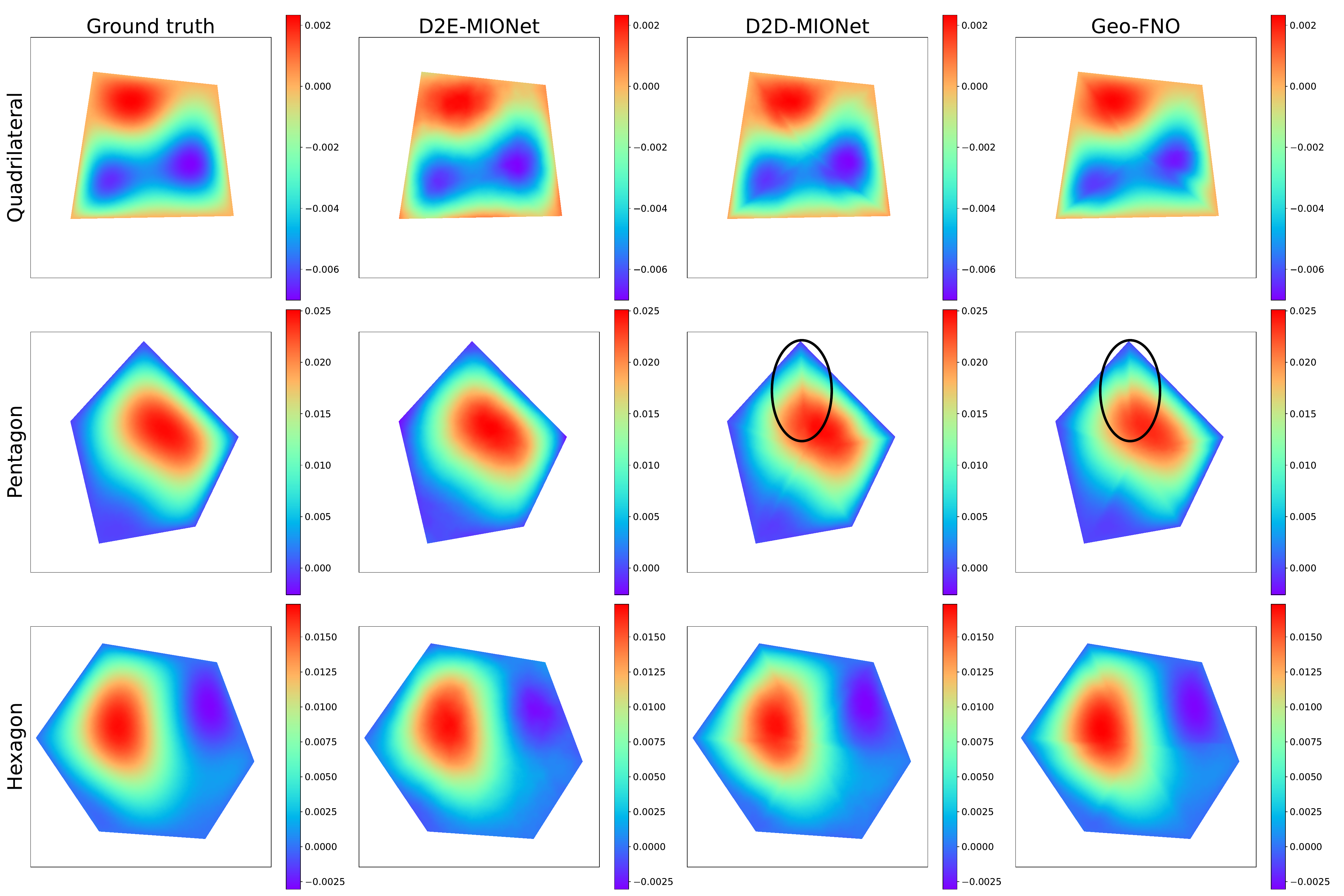}
\caption{The predictions of D2E-MIONet, D2D-MIONet, and Geo-FNO for Poisson equations on polygonal domains. The portion marked by the black circles shows that D2D-MIONet and Geo-FNO have bad regularity along the line segments connecting vertices to the centroid, while D2E-MIONet remains free from such problem.}
\label{fig:polygonal}
\end{figure}

In addition, we test the hybrid iterative method (HIM) \cite{hu2024hybrid} based on Gauss-Seidel (GS) iteration and the trained D2E-MIONet, since this model preserves linearity on the source term. The size of the mesh is approximately one million, and the correction period is set to 13K. The convergence speed of the HIM is 8.4 times that of the original GS method, as presented in Table \ref{tab:him}.

\begin{table}[htbp]
\centering
\begin{tabular}{|c|c|c|c|}
\hline
          & $\#$Iterations & Time (s)  & Speed up\\ \hline
GS        & 1142414   &   7780  &      /    \\ \hline
HIM & 121239  &  925  & $\times$8.4      \\ \hline
\end{tabular}
\caption{Comparison of the GS iterative method and the HIM based on GS and D2E-MIONet. The size of the mesh is approximately one million, and the correction period is set to 13K. The HIM achieves an acceleration effect of 8.4 times compared to the GS iteration.}
\label{tab:him}
\end{table}

\subsection{Smooth star domains with fully-parameterized solution mapping}\label{sec:smooth_star_domain}

Now we learn the fully-parameterized solution mapping, i.e. $(k_\Omega,f_\Omega,g_{\partial\Omega})\mapsto u_\Omega$, on smooth star domains. Using a similar strategy as previous subsection, we generate a dataset of size 2500, and learn this dataset by D2D-MIONet. An example of prediction is shown in Figure \ref{fig:smooth_2d}. It achieves 3.0\% relative $L^2$ error. We also test a 3-d case with a dataset of size 4000, and show an example of prediction in Figure \ref{fig:smooth_3d}. The relative $L^2$ error of the 3-d case is 12\%.

\begin{figure}[htbp]
\centering
\includegraphics[width=0.99\textwidth]{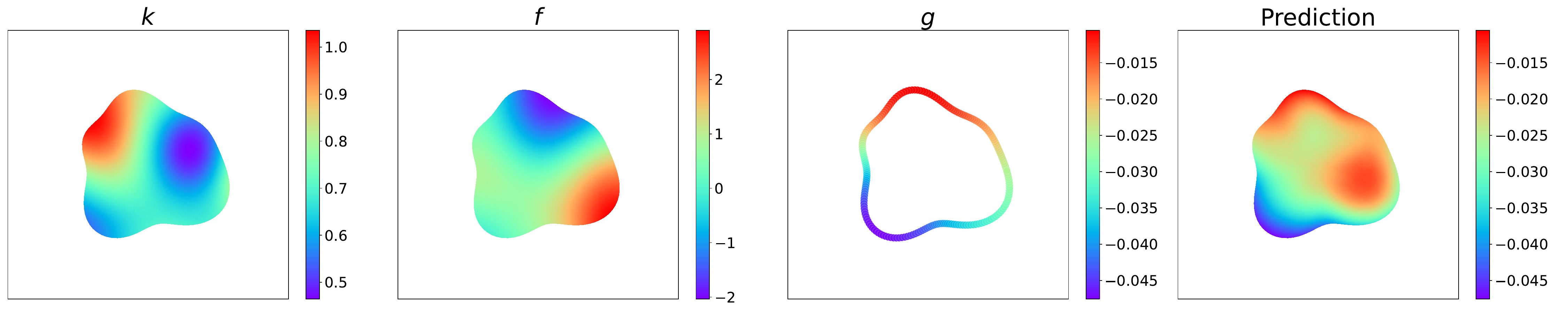}
\caption{An example of prediction on the smooth star domain with fully-parameterized solution mapping.}
\label{fig:smooth_2d}
\end{figure}

\begin{figure}[htbp]
\centering
\includegraphics[width=0.99\textwidth]{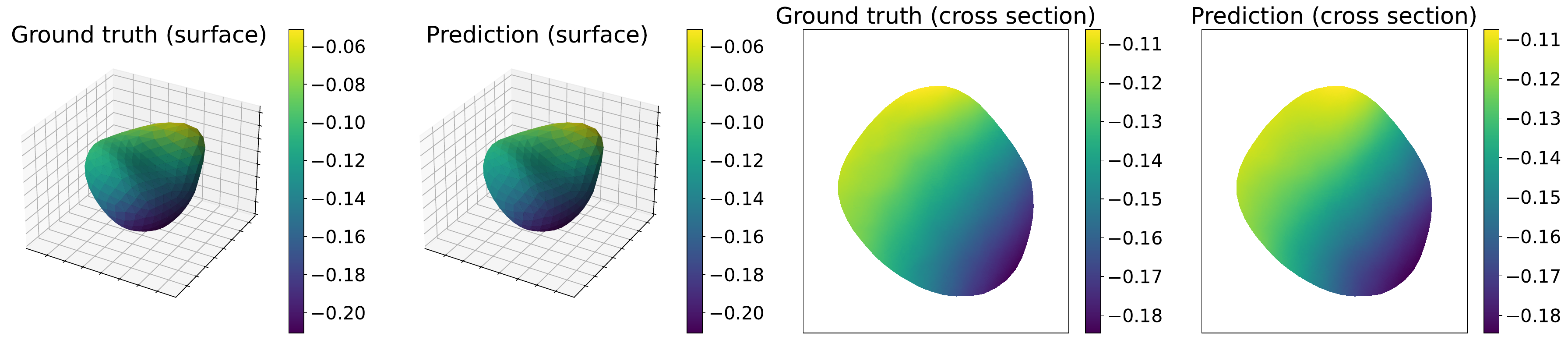}
\caption{An example of prediction in the three-dimensional space.}
\label{fig:smooth_3d}
\end{figure}

\subsection{Locally deformed domains}

In this subsection, we consider the mapping $f_\Omega\mapsto u_\Omega$, and study the domains where only local parts are changing. The setting of the domains follows Section \ref{sec:local_deformed}. In this case, a larger square is fixed while a smaller square is moving above the larger square. The deformation mappings are discontinuous on the contact edge of the two adjacent squares, and keep identity on the larger square. We generate a dataset of size 3800, and learn this dataset by D2E-MIONet. The relative $L^2$ error of the trained model is 1.9\%, and we display two examples of predictions in Figure \ref{fig:square}. As aforementioned, the predicted solutions of D2E-MIONet show good regularity, even on the contact edges.

\begin{figure}[htbp]
\centering
\includegraphics[width=0.99\textwidth]{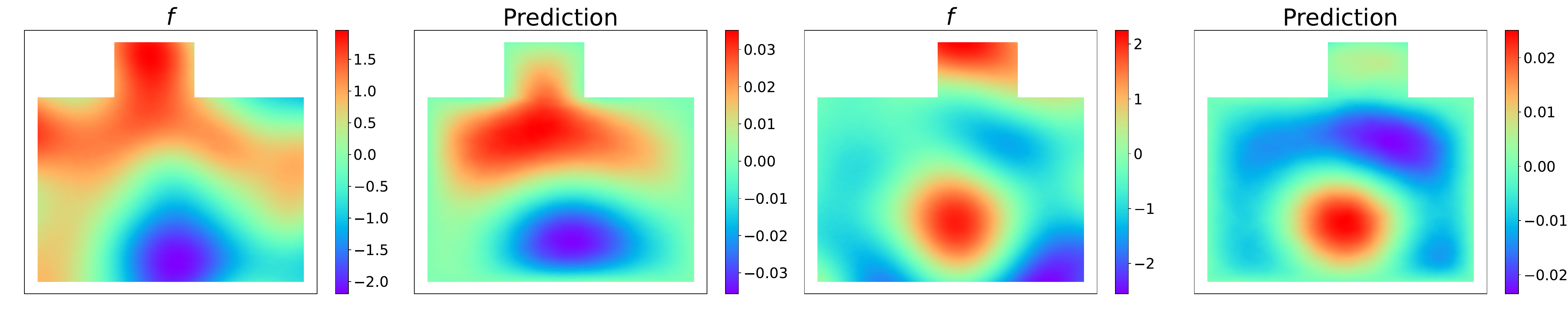}
\caption{Two examples of predictions on locally deformed domains.}
\label{fig:square}
\end{figure}

\subsection{Other examples}

Finally we present two more other examples for reference. The first one is the incompressible flow in a pipe which involves the incompressible Navier-Stokes equation. Adapting to the shape of the pipe regions, the standard domain in this experiment is chosen as a rectangle. The dataset is from \cite{li2022fourier}. The second example is the Poisson equation on the annular domains. Such a type of domains are generated as composite domains in Section \ref{sec:composite_domain}, with the standard circular ring as the reference domain. We present the predictions in Figure \ref{fig:others}.

\begin{figure}[htbp]
\centering
\includegraphics[width=0.99\textwidth]{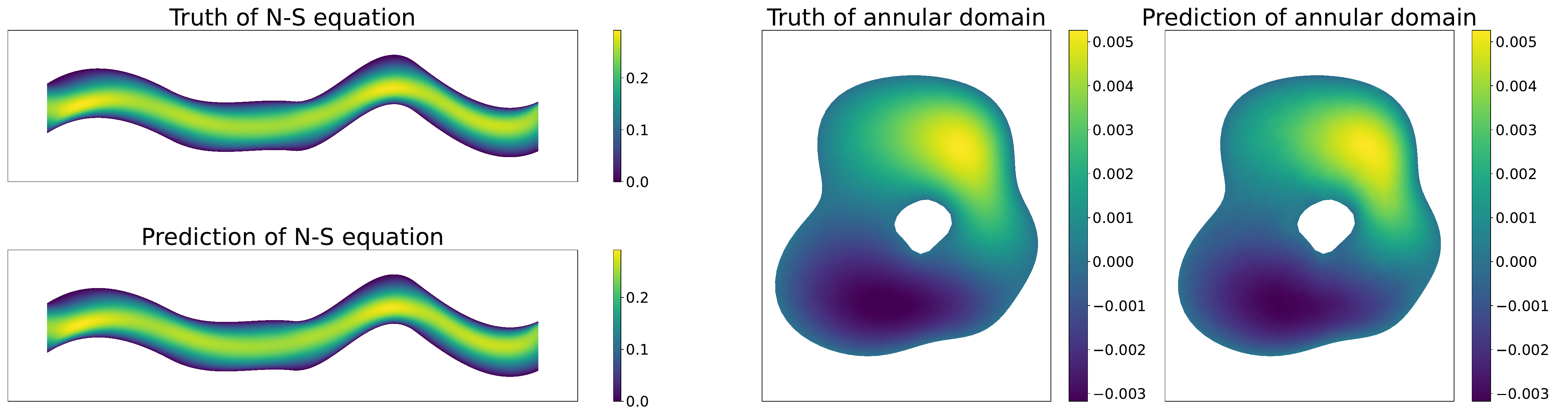}
\caption{(Left) The prediction on the pipe for N-S equation. (Right) The prediction on the annular domain for Poisson equation.}
\label{fig:others}
\end{figure}

\section{Conclusions}\label{sec:conclusions}

We have established a deformation-based framework for learning solution mappings of PDEs defined on varying domains. This framework starts by encoding the functions on different domains via a deformation system and a standard reference domain. According to the solution spaces that neural networks mapping to, two subframeworks are raised, called D2D (Deformation-based discretization to Deformed solution space) and D2E (Deformation-based discretization to Extended solution space). We built a complete convergence analysis framework for the D2D and the D2E subframeworks. Specifically, we presented a theoretical example of Poisson equation on star domains, and systematically proved its convergence property.

There are four important features of this framework: (1) The domains under consideration are not required to be diffeomorphic, therefore a wide range of regions can be covered by one model provided they
are homeomorphic. (2) The deformation mapping is unnecessary to be continuous, thus it can
be flexibly established via combining a primary identity mapping and a local deformation
mapping. This capability facilitates the resolution of large systems where only local parts of the geometry undergo change. (3) In fact, the recent methods (Geo-FNO, DIMON, etc.) belong to the D2D subframework. We point out that the D2D subframework introduces regularity issues, whereas the proposed D2E subframework remains free from such problems. From a comprehensive perspective, the D2E subframework is better. (4) If a linearity-preserving neural operator such as MIONet is adopted,
this framework still preserves the linearity of the surrogate solution mapping on its source term
for linear PDEs, thus it can be applied to the hybrid iterative method (HIM). 

Lastly we showed several numerical experiments to test the methods, mainly including Poisson equations on polygonal domains, smooth star domains, and locally deformed domains. In terms of performance, the D2D and the D2E subframeworks exhibit comparable test errors, while the D2E subframework predicts solutions with better regularity. We also adopted the trained D2E-MIONet to implement the HIM on the polygonal domain with unstructured mesh of size approximately one million, and achieved an acceleration effect of 8.4 times. 

As the deformation-based framework enables the derivative computation of the solution with respect to the shape, we expect to apply it to the shape optimization and the inverse design in the future.

\section*{Acknowledgments}

This research is supported by National Natural Science Foundation of China (Grant Nos. 12171466, 12271025 and 92470119).

\appendix

\section*{Appendix}

\section{Proof of Theorem \ref{thm:app_metric_to_banach}}\label{sec:proof_app_metric_to_banach}

This theorem is a general form of the approximation theorem of MIONet. Note that the proof for MIONet on Banach spaces in fact only utilizes the approximation result of Schauder basis, and it does not involve other properties of the Banach spaces. Such an approximation result can be replaced by Assumption \ref{ass:projection}. Here we simply show the key points.

Based on the injective tensor product \cite{ryan2002introduction}
\begin{equation}
    C(K,Y)= C(K)\itp Y,
\end{equation}
there exist some $f_j\in C(K)$ (can be extended to a continuous function on $X$ by Dugundji's theorem \cite{dugundji1951extension}) and $u_j\in Y$ such that
\begin{equation}
\norm{\mathcal{G}-\sum_{j=1}^p f_{j}\cdot u_j}_{C(K,Y)}<\epsilon.
\end{equation}
Assumption \ref{ass:projection} shows that there exist sufficiently large $q$, such that
\begin{equation}
\norm{\mathcal{G}-\sum_{j=1}^p f_{j}(P_{q}(\cdot))\cdot u_j}_{C(K,Y)}<\epsilon.
\end{equation}
Let $\tilde{\mathcal{G}}:=(f_1\circ\psi_{q},...,f_p\circ\psi_{q})$ and $u(\alpha_1,...,\alpha_p):=\sum_{i=1}^p\alpha_iu_i$, then we immediately obtain \eqref{eq:app_general}. \qed

\section{Multiple-input\&output case}\label{sec:mul_in_out}

Although the above discussion focuses on the single-input and single-output case ($f_\Omega$ to $u_\Omega$), we can easily generalize it to multiple-input and multiple-output scenarios. Assume that there are several input functions $f_\Omega^1,f_\Omega^2,...,f_{\Omega}^k$ (source term or coefficient functions related to the differential operator) defined on $\Omega$, $g_{\partial\Omega}^1,g_{\partial\Omega}^2,...,g_{\partial\Omega}^l$ (boundary conditions) defined on $\partial\Omega$, as well as several output solutions $u_\Omega^1,u_\Omega^2,...,u_\Omega^m$ defined on $\Omega$, their corresponding spaces are
\begin{equation}
    f_\Omega^i\in X=\bigsqcup_{\Omega\in U}\mathcal{B}(\Omega),\quad g_{\partial\Omega}^i\in Y=\bigsqcup_{\Omega\in U}\mathcal{B}(\partial\Omega),\quad u_\Omega^i\in X=\bigsqcup_{\Omega\in U}\mathcal{B}(\Omega),
\end{equation}
then the targeted continuous mapping is
\begin{equation}
\begin{split}
    \bar{X}&:=X^k\times Y^l\\
    \cup& \\
    \mathcal{H}: K&\longrightarrow  X^m,\\
    (f_{\Omega}^i,g_{\partial\Omega}^j)&\longmapsto (u_\Omega^1,...,u_\Omega^m)
\end{split}
\end{equation}
where $(f_{\Omega}^i,g_{\partial\Omega}^j)$ means $(f_{\Omega}^1,...,f_{\Omega}^k,g_{\partial\Omega}^1,...,g_{\partial\Omega}^l)$ for convenience. Provided that $X$ satisfies Assumption \ref{ass:metric}-\ref{ass:projection} with Definition \ref{def:X_metric} and discretization $\{\phi_n^X\}$, the metric for $Y$ can be similarly defined as
\begin{equation}
    d_Y(g_{\partial\Omega_1},g_{\partial\Omega_2}):=d_{U}(\Omega_1,\Omega_2)+\norm{g_{\partial\Omega_1}\circ\mathcal{D}[\Omega_1]|_{\partial\Omega_0}-g_{\partial\Omega_2}\circ\mathcal{D}[\Omega_2]|_{\partial\Omega_0}}_{L^2(\partial\Omega_0)},
\end{equation}
and we further assume that $Y$ satisfies Assumption \ref{ass:projection} with discretization $\{\phi_n^Y\}$. Consequently the metric for $\bar{X}$ can be given as
\begin{equation}
    d_{\bar{X}}((f_{\Omega_1}^i,g_{\partial\Omega_1}^j),(f_{\Omega_2}^i,g_{\partial\Omega_2}^j)):=\sum_{i=1}^kd_X(f_{\Omega_1}^i,f_{\Omega_2}^i)+\sum_{j=1}^ld_Y(g_{\partial\Omega_1}^j,g_{\partial\Omega_2}^j),
\end{equation}
together with the discretization
\begin{equation}
    \phi_n^{\bar{X}}((f_{\Omega}^i,g_{\partial\Omega}^j)):=(\phi_n^X(f_{\Omega}^1),...,\phi_n^X(f_{\Omega}^k),\phi_n^Y(g_{\partial\Omega}^1),...,\phi_n^Y(g_{\partial\Omega}^l)).
\end{equation}
Finally we derive a metric-to-Banach mapping
\begin{equation}
\begin{split}
    \bar{X}&\\
    \cup& \\
    \hat{\mathcal{H}}: K&\longrightarrow  \mathcal{B}(\Omega_0)^m\subset L^2(\Omega_0,\R^m),\quad\quad ({\rm D2D})\\
    (f_{\Omega}^i,g_{\partial\Omega}^j)&\longmapsto (u_\Omega^1\circ \mathcal{D}[\Omega],...,u_\Omega^m\circ \mathcal{D}[\Omega])
\end{split}
\end{equation}
or
\begin{equation}
\begin{split}
    \bar{X}&\\
    \cup& \\
    \hat{\mathcal{H}}: K&\longrightarrow  L^2(V,\R^m).\quad\quad ({\rm D2E})
\end{split}
\end{equation}
Such mapping can be learned by a neural network as in the single-input\&output case.

\section{Choices of neural operators}\label{sec:choices_no}

\textbf{MIONet.} D2D subframework with MIONet \cite{jin2022mionet} applied as the neural operator has been preliminarily studied in the previous version of this work \cite{xiao2024learning} as well as \cite{yin2024dimon}. In this case, $u(\alpha_1,...,\alpha_p):=\sum_{i=1}^p\alpha_iu_i$ is learned as the trunk net of MIONet, and the input signals of $\Omega$ and $f_\Omega\circ\mathcal{D}[\Omega]$ are separately handled by two different branch nets. The corresponding approximation theory is as follows.
\begin{theorem}[approximation theory for MIONet on metric space]\label{thm:approximation_MIONet}
    Let $X_{i}$ be metric spaces and $Y$ be a Banach space, assume that $X_{i}$ satisfies Assumption \ref{ass:projection} with the projection mapping $P_q^i=\psi_q^i\circ\phi_q^i$ and $\kappa_i(q)$, $K_{i}$ is a compact set in $X_{i}$. Suppose that
    \begin{equation}
    \mathcal{G}: K_{1}\times \cdots \times K_{n} \rightarrow Y   
    \end{equation}
    is a continuous mapping, then for any $\epsilon>0$, there exist positive integers $p_{i}, q_{i}$, continuous functions $g^{i}_{j}\in C(\mathbb{R}^{q_{i}})$ and $u_{j} \in Y$ such that 
    \begin{equation}
    \sup_{v_{i}\in K_{i}} \norm{\mathcal{G}(v_{1},\cdots,v_{n}) - \sum_{j = 1}^{p}g_{j}^{1}(\phi_{q_{1}}^{1}(v_{1}))\cdots g_{j}^{n}(\phi_{q_{n}}^{n}(v_{n}))\cdot u_{j}}_Y< \epsilon.
    \end{equation}
\end{theorem}
\begin{proof}
    This proof is not significantly different from the proof of Theorem \ref{thm:app_metric_to_banach}, by considering the injective tensor product
\begin{equation}
    C(K_1\times K_2\times\cdots\times K_n,Y)= C(K_1)\itp C(K_2)\itp\cdots\itp C(K_n)\itp Y.
\end{equation}
    Actually Theorem \ref{thm:app_metric_to_banach} just regards $K_1\times K_2\times\cdots\times K_n$ as a whole $K$.
\end{proof}
When using MIONet, the metric-to-Banach mapping $\hat{\mathcal{H}}_\theta$ is expressed as
\begin{equation}
    \hat{\mathcal{H}}_\theta(f_\Omega)=\sum_{j = 1}^{p}g_{j}^{1}\left(\phi_{q_{1}}^{1}(\Omega);\theta\right)\cdot g_{j}^{2}\left(\phi_{q_{2}}^{2}(f_\Omega\circ\mathcal{D}[\Omega]);\theta\right)\cdot u_{j}(\cdot;\theta),
\end{equation}
where $(g_1^1,...,g_p^1):\R^{\kappa_1(q_1)}\to\R^p$, $(g_1^2,...,g_p^2):\R^{\kappa_2(q_2)}\to\R^p$, $(u_1,...,u_p):\R^{d}\to\R^p$ are modeled by neural networks such as MLP, CNN and so on. It is noticed that, D2D\&D2E based on MIONet still preserve the linearity on $f_\Omega$, i.e.
\begin{equation}
\begin{split}
    &\hat{\mathcal{H}}_\theta(\alpha f_\Omega^1+\beta f_\Omega^2)\circ\mathcal{D}[\Omega]^{-1}\\=&\alpha\hat{\mathcal{H}}_\theta(f_\Omega^1)\circ\mathcal{D}[\Omega]^{-1}+\beta\hat{\mathcal{H}}_\theta(f_\Omega^2)\circ\mathcal{D}[\Omega]^{-1},\ \forall f_\Omega^1,f_\Omega^2\in\mathcal{B}(\Omega),\ \forall\alpha,\beta\in\R,\quad\quad ({\rm D2D})
\end{split}
\end{equation}
and
\begin{equation}
\begin{split}
    &\hat{\mathcal{H}}_\theta(\alpha f_\Omega^1+\beta f_\Omega^2)|_\Omega\\=&\alpha\hat{\mathcal{H}}_\theta(f_\Omega^1)|_\Omega+\beta\hat{\mathcal{H}}_\theta(f_\Omega^2)|_\Omega,\ \forall f_\Omega^1,f_\Omega^2\in\mathcal{B}(\Omega),\ \forall\alpha,\beta\in\R,\quad\quad ({\rm D2E})
\end{split}
\end{equation}
as long as $g_j^2$ are set to linear. The linearity-preserving property indicates that this method is compatible with the hybrid iterative method \cite{hu2024hybrid}. The multiple-input\&output case can be similarly derived.

\textbf{FNO.} D2D subframework with FNO \cite{li2020fourier} applied as the neural operator has been studied in \cite{li2022fourier}. In this case, FNO maps between functions on rectangular domains, consequently the standard domain $\Omega_0$ is also required to be rectangular. Different from MIONet, FNO adopts a series of fixed $u_1,...,u_p$, which can be regarded as the nodal functions of linear finite elements on standard uniform grids, and also set $u(\alpha_1,...,\alpha_p):=\sum_{i=1}^p\alpha_iu_i$. Due to the restriction of FNO on the latent domain, as well as FNO specializes in dealing with smooth functions rectangular distributed, there may be some singular points leading to bad results. For example, the center of the disc could be singular for FNO, hence FNO performs well on strip-like regions, but worse on disc-like regions.

\section{Proof of Lemma \ref{lem:discretization_for_U}}\label{sec:proof_discretization_for_U}

For the sake of symbol simplicity, we temporarily use $P_n$ instead of $P_n^1$ in this proof. Denote the centroid and the measure of $\Omega\subset \R^d$ as $c(\Omega)$ and $m(\Omega)$, respectively. For $\Omega\in K$, let $b_\Omega$ and $b_{P_n(\Omega)}^{c(\Omega)}$ be the boundary functions of $\Omega$ and $P_n(\Omega)$ (both centered at $c(\Omega)$) respectively. It is noticed that the domains in $K$ are uniformly bounded, we have
\begin{equation}
\begin{split}
    &m((\Omega\backslash P_n(\Omega))\cup(P_n(\Omega)\backslash\Omega))\\ =&\int_{(\Omega\backslash P_n(\Omega))\cup(P_n(\Omega)\backslash\Omega)}dx\\ =&\int_{S}\frac{1}{d}\left(\max\left(b_\Omega,b_{P_n(\Omega)}^{c(\Omega)}\right)^d-\min\left(b_\Omega,b_{P_n(\Omega)}^{c(\Omega)}\right)^d\right)dm_s \\ 
    \leq&C\cdot\sup_{e\in S}|b_\Omega(e)-b_{P_n(\Omega)}^{c(\Omega)}(e)| \\
    \leq&C\cdot L\cdot M_n,
\end{split}
\end{equation}
where $m_s$ is the spherical measure on $S$, $C>0$ is a constant independent of $\Omega$ and $n$, $L$ is an upper bound of the Lipschitz constants of domains in $U$, and $M_n$ satisfies Eq. \eqref{eq:lim_Mn}. It is not difficult to find that the measure of $\Omega\in K$ has a positive lower bound, denoted by $m_0>0$ (choose a $\delta$ small enough for each $\Omega$ such that $m(\Omega')>\frac{1}{2}m(\Omega)$ for any $\Omega'\in B(\Omega,\delta)$, then consider the open covering), as well as an upper bound $m_1>m_0$. Hence
\begin{equation}
    m(P_n(\Omega))=m(\Omega\cup P_n(\Omega))-m(\Omega\backslash P_n(\Omega))\geq m_0-CLM_n \geq \frac{1}{2}m_0
\end{equation}
for $n$ large enough. Let $M$ be an upper bound of $|x|$ for $x\in\Omega\in K$, then
\begin{equation}
\begin{split}
d_{E}(c(\Omega),c(P_n(\Omega)))=&\norm{\frac{\int_\Omega xdx}{m(\Omega)}-\frac{\int_{P_n(\Omega)}xdx}{m(P_n(\Omega))}}_2 \\
=&\norm{\frac{(m(P_n(\Omega))-m(\Omega))\int_\Omega xdx+m(\Omega)(\int_\Omega xdx-\int_{P_n(\Omega)} xdx)}{m(\Omega)m(P_n(\Omega))}}_2 \\
\leq&\frac{4CLMm_1}{m_0^2}\cdot M_n,
\end{split}
\end{equation}
for $n$ large enough. It immediately leads to
\begin{equation}\label{eq:centroid}
\lim_{n\to\infty}\sup_{\Omega\in K}d_{E}(c(\Omega),c(P_n(\Omega)))=0.
\end{equation}

Denote by $b_{P_n(\Omega)}$ the boundary function of $P_n(\Omega)$ (centered at $c(P_n(\Omega))$ compared to $b_{P_n(\Omega)}^{c(\Omega)}$ which is centered at $c(\Omega)$). Let $Q$ be the unique intersection point in
\begin{equation}
    \left\{c(\Omega)+t\left(b_{P_n(\Omega)}(e)\cdot e+c(P_n(\Omega))-c(\Omega)\right)|t\geq0\right\}\cap\partial\Omega,
\end{equation}
then
\begin{equation}
\begin{split}
|b_{\Omega}(e)-b_{P_n(\Omega)}(e)|=&\norm{b_{\Omega}(e)\cdot e-b_{P_n(\Omega)}(e)\cdot e}_2 \\
=&\big\|b_{\Omega}(e)\cdot e+c(\Omega)-Q \\
&+Q-(b_{P_n(\Omega)}(e)\cdot e+c(P_n(\Omega))) \\
&+c(P_n(\Omega))-c(\Omega)\big\|_2 \\
\leq&\norm{b_{\Omega}(e)\cdot e+c(\Omega)-Q}_2+LM_n+d_E(c(\Omega),c(P_n(\Omega))).
\end{split}
\end{equation}
Denote the angle between $Q-c(\Omega)$ and $b_{\Omega}(e)\cdot e$ as $\alpha$, then $0\leq\alpha\leq\pi$ and
\begin{equation}
\begin{split}
\cos(\alpha)=&\frac{b_{P_n(\Omega)}(e)^2+\norm{b_{P_n(\Omega)}(e)\cdot e+c(P_n(\Omega))-c(\Omega)}_2^2-d_E(c(\Omega),c(P_n(\Omega)))^2}{2b_{P_n(\Omega)}(e)\norm{b_{P_n(\Omega)}(e)\cdot e+c(P_n(\Omega))-c(\Omega)}_2}\\
\geq&1-C_1d_E(c(\Omega),c(P_n(\Omega)))^2,
\end{split}
\end{equation}
where $C_1>0$ is a constant. The last inequality is due to $b_{P_n(\Omega)}(e)$ has a lower bound. Thus we have $\lim_{n\to\infty}\alpha=0$. Consequently,
\begin{equation}
\begin{split}
\norm{b_{\Omega}(e)\cdot e+c(\Omega)-Q}_2^2=&b_{\Omega}(e)^2+\norm{Q-c(\Omega)}_2^2-2b_{\Omega}(e)\norm{Q-c(\Omega)}_2\cos(\alpha) \\
=&(b_{\Omega}(e)-\norm{Q-c(\Omega)}_2)^2+2b_{\Omega}(e)\norm{Q-c(\Omega)}_2(1-\cos(\alpha))\\
\leq&L^2M_n^2+C_2d_E(c(\Omega),c(P_n(\Omega)))^2,
\end{split}
\end{equation}
for a constant $C_2>0$. Subsequently, we obtain
\begin{equation}
\begin{split}
|b_{\Omega}(e)-b_{P_n(\Omega)}(e)|\leq&\sqrt{L^2M_n^2+C_2d_E(c(\Omega),c(P_n(\Omega)))^2}+LM_n+d_E(C(\Omega),C(P_n(\Omega))),
\end{split}
\end{equation}
and then
\begin{equation}\label{eq:boundary}
\lim_{n\to\infty}\sup_{\Omega\in K}\norm{b_{\Omega}-b_{P_n(\Omega)}}_{C(S)}=0.
\end{equation}
The equations \eqref{eq:centroid} and \eqref{eq:boundary} lead to the final result. \qed

\section{Proof of Lemma \ref{lem:discretization_for_B}}\label{sec:proof_discretization_for_B}

For the sake of symbol simplicity, we temporarily use $P_n$ instead of $P_n^2$ in this proof. We first fix a $f\in L^2(\Omega_0)$, for any $\epsilon>0$, there exists a smooth function with a compact support in $\Omega_0$, denoted by $f_c\in C_c^\infty(\Omega_0)$, such that $\norm{f-f_c}_{L^2(\Omega_0)}<\epsilon$. It is obvious to see that
\begin{equation}
    \lim_{n\to\infty}\norm{f_c-P_n(f_c)}_{L^2(\Omega_0)}=0,
\end{equation}
since the derivatives of $f_c$ are bounded and $\lim_{n\to\infty}M_n=0$. Subsequently, 
\begin{equation}
    \varlimsup_{n\to\infty}\norm{f-P_n(f_c)}_{L^2(\Omega_0)}\leq\varlimsup_{n\to\infty}\norm{f-f_c}_{L^2(\Omega_0)}+\varlimsup_{n\to\infty}\norm{f_c-P_n(f_c)}_{L^2(\Omega_0)}<\epsilon.
\end{equation}
Note that
\begin{equation}
\begin{split}
    \norm{f-P_n(f)}_{L^2(\Omega_0)}^2=&\sum_{i=1}^n\int_{\Omega_i^n}\left(f(x)-\frac{\int_{\Omega_i^n}f(y)dy}{m(\Omega_i^n)}\right)^2dx \\
    \leq&\sum_{i=1}^n\int_{\Omega_i^n}\left(f(x)-\frac{\int_{\Omega_i^n}f(y)dy}{m(\Omega_i^n)}\right)^2+\left(\frac{\int_{\Omega_i^n}g(y)-f(y)dy}{m(\Omega_i^n)}\right)^2dx \\
    =&\sum_{i=1}^n\int_{\Omega_i^n}\left(f(x)-\frac{\int_{\Omega_i^n}g(y)dy}{m(\Omega_i^n)}\right)^2dx \\
    =&\norm{f-P_n(g)}_{L^2(\Omega_0)}^2,\quad\forall g\in L^2(\Omega_0),
\end{split}
\end{equation}
which means $P_n(f)$ is the projection of $f$ on the $n$-dimensional subspace and it is closest to $f$ in $L^2(\Omega_0)$ norm. Therefore
\begin{equation}
    \varlimsup_{n\to\infty}\norm{f-P_n(f)}_{L^2(\Omega_0)}\leq\varlimsup_{n\to\infty}\norm{f-P_n(f_c)}_{L^2(\Omega_0)}<\epsilon.
\end{equation}
Let $\epsilon\to 0$, we obtain $\lim_{n\to\infty}\norm{f-P_n(f)}_{L^2(\Omega_0)}=0$.

Next we show the uniform convergence on $K$. Given any $\epsilon>0$, due to the open covering $\cup_{f\in K}B(f,\epsilon)\supset K$, there exist $\{f_1,...,f_k\}$ such that $\cup_{i=1}^kB(f_i,\epsilon)\supset K$. Then for any $f\in K$, we can find a $f_l$ satisfying $\norm{f_l-f}_{L^2(\Omega_0)}<\epsilon$, consequently
\begin{equation}
    \norm{f-P_n(f)}_{L^2(\Omega_0)}\leq\norm{f-P_n(f_l)}_{L^2(\Omega_0)}\leq\epsilon+\max_{1\leq i\leq k}\norm{f_i-P_n(f_i)}_{L^2(\Omega_0)},
\end{equation}
which leads to
\begin{equation}
    \varlimsup_{n\to\infty}\sup_{f\in K}\norm{f-P_n(f)}_{L^2(\Omega_0)}\leq\epsilon+\varlimsup_{n\to\infty}\max_{1\leq i\leq k}\norm{f_i-P_n(f_i)}_{L^2(\Omega_0)}=\epsilon.
\end{equation}
Let $\epsilon\to 0$, we obtain $\lim_{n\to\infty}\sup_{f\in K}\norm{f-P_n(f)}_{L^2(\Omega_0)}=0$. \qed

\section{Proof of Theorem \ref{thm:continuity_poisson}}\label{sec:proof_continuity_poisson}

The main tool for this proof is the Green's function \cite{lax2005existence} as well as its uniform convergence on planar domains \cite{kalmykov2019uniform}. We first choose a domain $V=[-R,R]^2\subset\R^2$ large enough so that all the $\Omega\in K_1:=\pi_1(\sigma(K))$ are covered by $V$. Every function involved in this proof is considered as a function defined on $V$, which is extended by 0 outside of the domain. Let $f_{\Omega_n}\subset K$ be a series of functions converging to $f_\Omega\in K$ in $d_X$, then $\Omega_n\rightarrow\Omega$ in $d_U$, and $f_{\Omega_n}\circ\mathcal{D}[\Omega_n]\rightarrow f_\Omega\circ\mathcal{D}[\Omega]$ in $L^2(\Omega_0)$. Suppose that $G_n(x,y)$ and $G(x,y)$ are the Green's functions for $\Omega_n$ and $\Omega$ respectively. Next we present the proof through several steps.

    1. We will show $f_{\Omega_n}\to f_\Omega$ in $L^2(V)$. For any compact $J\subset \Omega\backslash\{c_\Omega\}$, it can be readily checked that $\mathcal{D}[\Omega_n]^{-1}(y)\to\mathcal{D}[\Omega]^{-1}(y)$ uniformly on $J$, subsequently
    \begin{equation}
    \begin{split}
        &\norm{f_{\Omega_n}-f_{\Omega_n}\circ\mathcal{D}[\Omega_n]\circ\mathcal{D}[\Omega]^{-1}}_{L^2(J)}^2\\ =&\int_J\left|f_{\Omega_n}\circ\mathcal{D}[\Omega_n](\mathcal{D}[\Omega_n]^{-1}(y))-f_{\Omega_n}\circ\mathcal{D}[\Omega_n](\mathcal{D}[\Omega]^{-1}(y))\right|^2dy\to 0{\rm\ \ as\ \ }n\to\infty,
    \end{split}
    \end{equation}
    since $\{f_{\Omega_n}\circ\mathcal{D}[\Omega_n]\}\subset K_2:=\pi_2(\sigma(K))$ is uniformly equicontinuous by Arzelà–Ascoli theorem. The arbitrariness of $J$ and the boundedness of $\{f_{\Omega_n}\}$ leads to \begin{equation}\label{eq:convergence_for_f_1}
        \norm{f_{\Omega_n}-f_{\Omega_n}\circ\mathcal{D}[\Omega_n]\circ\mathcal{D}[\Omega]^{-1}}_{L^2(V)}\to 0{\rm\ \ as\ \ }n\to\infty.
    \end{equation}
    Additionally,
    \begin{equation}
    \begin{split}
        &\norm{f_{\Omega_n}\circ\mathcal{D}[\Omega_n]\circ\mathcal{D}[\Omega]^{-1}-f_\Omega}_{L^2(\Omega)}^2\\ =&\int_{\Omega}\left|f_{\Omega_n}\circ\mathcal{D}[\Omega_n](\mathcal{D}[\Omega]^{-1}(y))-f_{\Omega}\circ\mathcal{D}[\Omega](\mathcal{D}[\Omega]^{-1}(y))\right|^2dy\\
        =&\int_{0}^{2\pi}\int_{0}^{1}\left|f_{\Omega_n}\circ\mathcal{D}[\Omega_n](r\cos(\theta),r\sin(\theta))-f_{\Omega}\circ\mathcal{D}[\Omega](r\cos(\theta),r\sin(\theta))\right|^2rb_\Omega^2(\theta)drd\theta \\
        \leq&C\int_{0}^{2\pi}\int_{0}^{1}\left|f_{\Omega_n}\circ\mathcal{D}[\Omega_n](r\cos(\theta),r\sin(\theta))-f_{\Omega}\circ\mathcal{D}[\Omega](r\cos(\theta),r\sin(\theta))\right|^2rdrd\theta \\
        =&C\norm{f_{\Omega_n}\circ\mathcal{D}[\Omega_n]-f_\Omega\circ\mathcal{D}[\Omega]}_{L^2(\Omega_0)}^2\to 0{\rm\ \ as\ \ }n\to\infty,
    \end{split}
    \end{equation}
    therefore
    \begin{equation}\label{eq:convergence_for_f_2}
        \norm{f_{\Omega_n}\circ\mathcal{D}[\Omega_n]\circ\mathcal{D}[\Omega]^{-1}-f_\Omega}_{L^2(V)}\to 0{\rm\ \ as\ \ }n\to\infty.
    \end{equation}
    By Eq. \eqref{eq:convergence_for_f_1} and \eqref{eq:convergence_for_f_2}, we derive that
    $f_{\Omega_n}\to f_\Omega$ in $L^2(V)$.

    2. We will show $u_{\Omega_n}(x)\to u_\Omega(x)$ for a.e. $x\in V$. Theorem 3.1 in \cite{kalmykov2019uniform} tells that $G_n(x,\cdot)$ uniformly converges to $G(x,\cdot)$ on $V\backslash\{x\}$, thus $G_n(x,\cdot)\to G(x,\cdot)$ in $L^2(V)$ for every $x\in \Omega$. By the representation formula using Green's function,
    \begin{equation}
        u_{\Omega_n}(x)=\int_{V}f_{\Omega_n}(y)G_n(x,y)dy\to\int_{V}f_{\Omega}(y)G(x,y)dy=u_\Omega(x){\rm\ \ as\ \ }n\to\infty,\quad\forall x\in\Omega.
    \end{equation}
    Moreover, $u_{\Omega_n}(x)$ converges to $0=u_\Omega(x)$ for every $x\in V\backslash\overline{\Omega}$, so that $u_{\Omega_n}(x)\to u_\Omega(x)$ for a.e. $x\in V$.

    3. We will show $u_{\Omega_n}\to u_\Omega$ in $L^2(V)$. The Green's functions are written as
    \begin{equation}
        G_n(x,y)=\Phi(y-x)+\phi_n^x(y),\quad G(x,y)=\Phi(y-x)+\phi^x(y),
    \end{equation}
    where $\Phi(y)=-\frac{1}{2\pi}\ln(|y|)$ is the fundamental solution, $\phi_n^x(y)$ and $\phi^x(y)$ are the corrector functions solving
\begin{equation}
	\begin{cases}
		-\Delta \phi_n^x=0 &\mbox{in} \; \Omega_n
		\\
		\phi_n^x = \Phi(y-x) &\mbox{on} \; \partial\Omega_n,
	\end{cases}\quad\quad 
    \begin{cases}
		-\Delta \phi^x=0 &\mbox{in} \; \Omega
		\\
		\phi^x = \Phi(y-x) &\mbox{on} \; \partial\Omega.
	\end{cases}
\end{equation}
The strong maximum principle indicates that the harmonic function $\phi_n^x$ attains its maximum and minimum on the boundary, therefore
\begin{equation}
\begin{split}
    |u_{\Omega_n}(x)|=&\left|\int_{\Omega_n}f_{\Omega_n}(y)G_n(x,y)dy\right|\\ 
    \leq&\int_{\Omega_n}M_f(|\Phi(y-x)|+|\phi_n^x(y)|)dy\\
    \leq&C_1+C_2\cdot|\ln({\rm dist}(x,\partial\Omega_n))|=:g_n(x),\quad x\in \Omega_n,
\end{split}
\end{equation}
where $M_f$ is an upper bound of $|f_{\Omega_n}|$, $C_1,C_2>0$ are two constants. Denote a neighborhood of $\partial\Omega_n$ by $E_n^\epsilon:=\{x\in \Omega_n|{\rm dist}(x,\partial\Omega_n)<\epsilon\}$, then 
\begin{equation}
\begin{split}
    \int_{\Omega_n}|\ln({\rm dist}(x,\partial\Omega_n))|^2dx=&\int_{E_n^\epsilon}|\ln({\rm dist}(x,\partial\Omega_n))|^2dx+\int_{\Omega_n\backslash E_n^\epsilon}|\ln({\rm dist}(x,\partial\Omega_n))|^2dx\\
    =&\int_{0}^{\epsilon}|\ln(r)|^2S_n(r)dr+\int_{\Omega_n\backslash E_n^\epsilon}|\ln({\rm dist}(x,\partial\Omega_n))|^2dx\\
    <&+\infty,
\end{split}
\end{equation}
where $S_n(r)$ denoting the length of curve $\{x\in\Omega_n|{\rm dist}(x,\partial\Omega_n)=r\}$. Consequently, $g_n\in L^2(V)$ by zero-expansion. Similarly, let $g(x):=C_1+C_2\cdot|\ln({\rm dist}(x,\partial\Omega))|$ for $x\in\Omega$, then $g\in L^2(V)$ by zero-expansion. Consider the neighborhood of $\partial\Omega$ denoted by $E^\epsilon:=\{x\in V|{\rm dist}(x,\partial\Omega)<\epsilon\}$, then
\begin{equation}
\begin{split}
    \norm{g_n-g}_{L^2(V)}^2=&\int_{V}(g_n(x)-g(x))^2dx\\
    =&\int_{E^\epsilon}(g_n(x)-g(x))^2dx+\int_{V\backslash E^\epsilon}(g_n(x)-g(x))^2dx\\
    \leq&2\int_{E^\epsilon}g_n(x)^2dx+2\int_{E^\epsilon}g(x)^2dx+\int_{V\backslash E^\epsilon}(g_n(x)-g(x))^2dx \\
    \leq&2\int_{E_n^{2\epsilon}}g_n(x)^2dx+2\int_{E^\epsilon}g(x)^2dx+\int_{V\backslash E^\epsilon}(g_n(x)-g(x))^2dx\\
\end{split}
\end{equation}
for $n$ large enough. We can check that $g_n(x)$ converges to $g(x)$ uniformly on $V\backslash E^\epsilon$, therefore $\int_{V\backslash E^\epsilon}(g_n(x)-g(x))^2dx\to 0$ as $n\to\infty$. Note that
\begin{equation}
    \int_{E_n^{2\epsilon}}g_n(x)^2dx=\int_{0}^{2\epsilon}(C_1+C_2|\ln(r)|)^2S_n(r)dr\leq C_3(\epsilon\ln^2(2\epsilon)-\epsilon\ln(2\epsilon)+\epsilon),
\end{equation}
and also
\begin{equation}
    \int_{E^{\epsilon}}g(x)^2dx=\int_{0}^{\epsilon}(C_1+C_2|\ln(r)|)^2S(r)dr\leq C_3(\epsilon\ln^2(\epsilon)-\epsilon\ln(\epsilon)+\epsilon),
\end{equation}
subsequently
\begin{equation}
    \varlimsup_{n\to\infty}\norm{g_n-g}_{L^2(V)}^2\leq C_4\cdot(\epsilon\ln^2(\epsilon)-\epsilon\ln(\epsilon)+\epsilon).
\end{equation}
Let $\epsilon\to 0$ we obtain $g_n\to g$ in $L^2(V)$. Up to now we have pointed out that
\begin{itemize}
    \item $u_{\Omega_n}(x)\to u_\Omega(x)$ for a.e. $x\in V$,
    \item $|u_{\Omega_n}(x)|\leq g_n(x)$ for $x\in V$,
    \item $g_n\to g$ in $L^2(V)$,
\end{itemize}
which lead to $u_{\Omega_n}\to u_\Omega$ in $L^2(V)$ by Vitali convergence theorem.

4. It is noticed that $u_\Omega$ is uniformly continuous on $V$, and $\mathcal{D}[\Omega_n](x)\to\mathcal{D}[\Omega](x)$ uniformly on $\Omega_0$, so that $u_\Omega\circ\mathcal{D}[\Omega_n]\to u_\Omega\circ\mathcal{D}[\Omega]$ in $L^2(\Omega_0)$. Moreover,
\begin{equation}
\begin{split}
    &\norm{u_{\Omega_n}\circ\mathcal{D}[\Omega_n]-u_{\Omega}\circ\mathcal{D}[\Omega_n]}_{L^2(\Omega_0)}^2\\=&\int_{\Omega_0}\left|u_{\Omega_n}(\mathcal{D}[\Omega_n](x))-u_{\Omega}(\mathcal{D}[\Omega_n](x))\right|^2dx\\
    =&\int_{0}^{2\pi}\int_0^{b_{\Omega_n}(\theta)}\left|u_{\Omega_n}(c_{\Omega_n}+r\cdot(\cos(\theta),\sin(\theta)))-u_{\Omega}(c_{\Omega_n}+r\cdot(\cos(\theta),\sin(\theta)))\right|^2b_{\Omega_n}^{-2}(\theta)rdrd\theta\\
    \leq&C_5\int_{0}^{2\pi}\int_0^{b_{\Omega_n}(\theta)}\left|u_{\Omega_n}(c_{\Omega_n}+r\cdot(\cos(\theta),\sin(\theta)))-u_{\Omega}(c_{\Omega_n}+r\cdot(\cos(\theta),\sin(\theta)))\right|^2rdrd\theta\\
    =&C_5\norm{u_{\Omega_n}-u_\Omega}_{L^2(\Omega_n)}^2\leq C_5\norm{u_{\Omega_n}-u_\Omega}_{L^2(V)}^2,
\end{split}
\end{equation}
therefore $u_{\Omega_n}\circ\mathcal{D}[\Omega_n]\to u_{\Omega}\circ\mathcal{D}[\Omega_n]$ in $L^2(\Omega_0)$. Finally, we derive that $u_{\Omega_n}\circ\mathcal{D}[\Omega_n]\to u_{\Omega}\circ\mathcal{D}[\Omega]$ in $L^2(\Omega_0)$, the proof is completed. \qed

\bibliographystyle{abbrv}
\bibliography{references}

\end{document}